\documentclass[final,3p]{elsarticle}
 \usepackage{graphics}
 \usepackage{graphicx}
 \usepackage{epsfig}
\usepackage{amssymb}
 \usepackage{amsthm}
 \usepackage{lineno}
 \usepackage{amsmath}
   \numberwithin{equation}{section}
\usepackage{mathrsfs}

\newtheorem{thm}{Theorem}[section]

\newtheorem{lem}[thm]{Lemma}
\newtheorem{prop}[thm]{Proposition}
\newtheorem{defn}[thm]{Definition}
\newtheorem{rem}[thm]{Remark}

\biboptions{sort&compress,square}
\journal{}
\begin{document}
\begin{frontmatter}
\author{Jian Wang}
\ead{wangj068@gmail.com}
\author{Yong Wang\corref{cor2}}
\cortext[cor2]{Corresponding author. \textbf{Email address: wangy581@nenu.edu.cn (Yong Wang)}}
\address{School of Mathematics and Statistics, Northeast Normal University,
Changchun, 130024, P.R.China}
\title{Noncommutative Residue and sub-Dirac Operators for Foliations}
\begin{abstract}
In this paper, we define  lower dimensional volumes associated to sub-Dirac operators for foliations. In some cases, we compute these
 lower dimensional volumes. We also prove the Kastler-Kalau-Walze type theorems for foliations with or without boundary. As a corollary,
 we give an explanation of the gravitational action for the Robertson-Walker space $[a, b]\times_{f}M^{3}$.
\end{abstract}
\end{frontmatter}
\section{Introduction}
\label{1}
 The noncommutative residue, found in \cite{Gu}, \cite{Wo1} and \cite{Wo2}, plays a prominent role in noncommutative geometry. In \cite{Co1},
  Connes used the noncommutative residue to derive a conformal 4-dimensional Polyakov action analogy. Moreover, in \cite{Co2},  Connes proved that the noncommutative
residue on a compact manifold $M$ coincided with the Dixmier's trace on pseudodifferential operators of order $-{\rm {dim}}M$. Several
years ago, Connes made a challenging observation that the noncommutative residue of the square of the inverse of the Dirac
operator was proportional to the Einstein-Hilbert action, which was called Kastler-Kalau-Walze Theorem now.  Kastler gave a brute-force
proof of this theorem \cite{Ka}. In \cite{KW}, Kalau and Walze also gave a proof of this theorem by using normal coordinates .
In \cite{RP}, Ponge explained how to define ``lower dimensional" volumes of any compact Riemannian manifold as the integrals of local Riemannian
invariants and dealt with the lower dimensional volumes in even dimension. For spin manifolds with boundary and the associated
Dirac operators, Wang defined and computed lower dimensional volumes and got a Kastler-Kalau-Walze type theorem in \cite{Wa4}, \cite{Wa3} and
\cite{Wa2}. In \cite{LW}, Liu and Wang derived a Kastler-Kalau-Walze theorem for foliations. In \cite{WW}, we got a Kastler-Kalau-Walze type
theorem associated to nonminimal operators by heat equation asymptotics on compact manifolds without boundary.

The warped product $[a, b]\times_{f}M^{3}$ with the metric $dt^{2}+f(t)^{2}g^{TM}$ is an important space in physics. Here $M$ maybe is not spin.
One of the motivations is to give a Kastler-Kalau-Walze type theorem  for this manifold with boundary. We note that $[a, b]\times_{f}M^{3}$ is
a special foliation with spin leave $[a,b]$. Since $[a, b]\times_{f}M^{3}$ is not spin, we consider sub-Dirac operators for foliations with
spin leave instead of Dirac operators. In this paper, we define  lower dimensional volumes associated to sub-Dirac operators for foliations.
In some cases, we compute these lower dimensional volumes. We also prove the Kastler-Kalau-Walze type theorems for foliations with or without
 boundary. As a corollary, we give an explanation of the gravitational action for the Robertson-Walker space $[a, b]\times_{f}M^{3}$.

This paper is organized as follows: In Section 2, we recall the sub-Dirac operators and define the lower dimensional volumes associated to
sub-Dirac operators for foliation with spin leave. In Section 3, for 4-dimensional compact foliations with boundary and the associated sub-Dirac
 operators, we compute the lower dimensional volumes $Vol_{4}^{(1,1)}$, $Vol_{3}^{(1,1)}$ and get the Kastler-Kalau-Walze type theorems
  in these case. In Section 4, we compute the lower dimensional volume $Vol_{6}^{(2,2)}$ associated to sub-Dirac operators for foliations.
  In section 5, we compute the lower dimensional volume $Vol_{5}^{(2,2)}$ associated to sub-Dirac operators for foliations.
  In section 6, we compute the lower dimensional volumes and the spectral action for the Robertson-Walker space $[a, b]\times_{f}M^{3}$.

\section{Lower-Dimensional Volumes associated to sub-Dirac operators for Foliations}
In this section, we shall restrict our attention to the sub-Dirac operators for foliations. Let $(M,F)$ be a closed foliation and $M$ has
spin leave, $g^F$ be a metric on  $F$. Let $g^{TM}$ be a metric on $TM$ which restricted to $g^F$ on $F$. Let $F^\perp$ be the orthogonal
complement of $F$ in $TM$ with
 respect to $g^{TM}$. Then we have the following orthogonal splitting
 \begin{eqnarray}
&&TM=F\oplus F^\perp,\nonumber\\
&&g^{TM}=g^F\oplus g^{F^\perp},
\end{eqnarray}
 where $g^{F^\perp}$ is the restriction of $g^{TM}$ to $F^\perp$.

 Let $P, P^\perp$ be the orthogonal projection from $TM$ to $F$, $F^\perp$
 respectively. Let $\nabla^{TM}$ be the Levi-Civita connection of
 $g^{TM}$ and $\nabla^F$ (resp. $\nabla^{F^\perp})$ be the
 restriction of $\nabla^{TM}$ to $F$ (resp. $F^\perp$).
 Without loss of generality, we assume $F$ is oriented, spin and carries a fixed spin
structure. Furthermore, we assume $F^\perp$ is oriented and we do not assume that ${\rm dim}F$ and ${\rm dim}F^{\perp}$ are even.
 By assumption,  we may write
\begin{eqnarray}
&&\nabla^F=P\nabla^{TM}P,\nonumber\\
&&\nabla^{F^\perp}=P^\perp\nabla^{TM}P^\perp.
\end{eqnarray}

Let $S(F)$ be the bundle of spinors associated to $(F,g^F)$. For any $X\in \Gamma(F),$ denote by $c(X)$ the Clifford
 action of $X$ on $S(F)$. The exterior algebra bundle of $F^{\perp}$ is defined by $\wedge(F^{\perp,\star})$.
Then $\wedge(F^{\perp,\star})$ carries a canonically induced metric $g^{\wedge(F^{\perp,\star})}$ from $g^{F^\perp}$. For
any $U\in \Gamma(F^\perp)$, let $U^*\in \Gamma(F^{\perp,*})$  be the corresponding dual of $U$ with respect to $g^{F^\perp}$.
The Clifford  action of $U$ is defined by
\begin{eqnarray}
&&c(U)=U^*\wedge-i_U,\nonumber\\
&&\widehat{c}(U)=U^*\wedge+i_U,
\end{eqnarray}
where $U^*\wedge$ and $i_U$ are the exterior and inner
 multiplications.

Let $S(F)\otimes\wedge(F^{\perp,\star})$ be the tensor product of $S(F)$ and $\wedge(F^{\perp,\star})$.
For $X\in \Gamma(F),~U\in \Gamma(F^\perp)$, the operators $c(X),~c(U)$ and $\widehat{c}(U)$ are anticommute which extend naturally to
$S(F)\otimes\wedge(F^{\perp,\star})$. For $s_{1}\in S(F)$ and $s_{2}\in \wedge(F^{\perp,\star})$ , we assume that
\begin{eqnarray}
&&\big(c(X)c(U)\big)(s_{1}\otimes s_{2})=c(X)s_{1}\otimes c(U)s_{2};\nonumber\\
&&\big(c(U)c(X)\big)(s_{1}\otimes s_{2})=-c(X)s_{1}\otimes c(U)s_{2}.
\end{eqnarray}

Moreover, the connections $\nabla^F ~(\nabla^{F^\perp})$ lift to $S(F)$ ($\wedge(F^{\perp,\star})$) naturally denoted by
$\nabla^{S(F)}$ ($\nabla^{\wedge(F^{\perp,\star})}$) respectively. Then $S(F)\otimes\wedge(F^{\perp,\star})$ carries the induced
tensor product connection
 \begin{equation}
 \nabla^{S(F)\otimes\wedge(F^{\perp,\star})}= \nabla^{S(F)}\otimes \texttt{Id}_{\wedge(F^{\perp,\star})}
   +\texttt{Id}_{S(F)}\otimes \nabla^{\wedge(F^{\perp,\star})}.
\end{equation}
Then we can define
$S\in \Omega(T^*M)\otimes \Gamma({\rm End}(TM))$
\begin{equation}
\nabla^{TM}=\nabla^{F}+\nabla^{F^\perp}+S.
\end{equation}
For any $X\in \Gamma(TM)$, $S(X)$ exchanges $\Gamma(F)$ and $\Gamma(F^\perp)$ and is skew-adjoint with respect to $g^{TM}$.
Let $\{f_i\}_{i=1}^{p}$ be an oriented orthonormal basis of $F$,
we define
\begin{equation}
\widetilde{\nabla}^{F}=\nabla^{S(F)\otimes\wedge(F^{\perp,\star})}
+ \frac{1}{2}\sum_{j=1}^{p}\sum_{s=1}^q<S(.)f_j,h_s>c(f_j)c(h_s).
\end{equation}
where the vector bundle $F^\perp$ might well be non-spin. An application of definition 2.2 in \cite{LZ} shows that
the following sub-Dirac operator.

\begin{defn}
Let $D_{F}$ be the operator mapping
from $\Gamma(S(F)\otimes\wedge(F^{\perp,\star}))$ to itself defined by
\begin{equation}
D_{F}=\sum_{i=1}^{p}c(f_i)\widetilde{\nabla}^{F}_{f_i}+\sum_{s=1}^{q}c(h_s)\widetilde{\nabla}^{F}_{h_s}.
\end{equation}
\end{defn}
From (2.19) in \cite{LZ},  we shall make use of the Bochner Laplacian $\triangle^{F}$ stating that
\begin{equation}
\triangle^{F}:=-\sum_{i=1}^{p}(\widetilde{\nabla}^{F}_{f_i})^2
-\sum_{s=1}^{q}(\widetilde{\nabla}^{F}_{h_s})^2+
\widetilde{\nabla}^{F}_{\sum_{i=1}^{p}\nabla^{TM}_{f_i}f_i} +
\widetilde{\nabla}^{F}_{\sum_{s=1}^{q}\nabla^{TM}_{h_s}h_s}.
\end{equation}
Let $r_M$ be the scalar curvature of the metric $g^{TM}$. Let
 $R^{F^\bot}$ be the curvature tensor of $F^\bot$.
From Theorem 2.3 in \cite{LZ}, we have the following Lichnerowicz formula for $D_{F}$.

 \begin{thm}\cite{LZ}
 The following identity holds
\begin{eqnarray}
D^2_{F}&=&\triangle^{F}+\frac{r_M}{4}+\frac{1}{4}\sum_{i=1}^{p}\sum_{r,s,t=1}^q\left<R^{F^\bot}(f_i,h_r)h_t,h_s\right>
c(f_i)c(h_r)\widehat{c}(h_s)\widehat{c}(h_t)\nonumber\\
&&+\frac{1}{8}\sum_{i,j=1}^{p}\sum_{s,t=1}^q\left<R^{F^\bot}(f_i,f_j)h_t,h_s\right>
c(f_i)c(f_j)\widehat{c}(h_s)\widehat{c}(h_t)\nonumber\\
&&+\frac{1}{8}\sum_{s,t,r,u=1}^q\left<R^{F^\bot}(h_r,h_l)h_t,h_s\right>
c(h_r)c(h_u)\widehat{c}(h_s)\widehat{c}(h_t).
\end{eqnarray}
\end{thm}

In order to get a Kastler-Kalau-Walze type theorem for foliations, Liu and Wang \cite{LW} considered the noncommutative residue  of
the $-n+2$ power of the sub-Dirac operator, and got the following Kastler-Kalau-Walze type theorem for foliations.

\begin{thm}\cite{LW}
Let $(M^{n},F)$ be  a compact even-dimensional oriented foliation with spin leave and codimension $q$, and $D_{F}$ be the sub-Dirac operator,
then $ \lim _{\varepsilon\rightarrow 0}\varepsilon^{\frac{q}{2}}\texttt{Res}(D^{-n+2}_{F,\varepsilon})$ is proportional to
$\int_{M}[k^{F}+\Phi(\omega)]dvol_{g}$.
\end{thm}
Similarly, we have
\begin{thm}
Let $(M^{n},F)$ be  a compact even-dimensional oriented foliation with spin leave and codimension $q$, and $D_{F}$ be the sub-Dirac operator,
then
 \begin{equation}
Wres(D_{F}^{-n+2})=\tilde{c}_{0}\int_{M}r_{M}\texttt{dvol}_{g},
\end{equation}
where $\tilde{c}_{0}=-\frac{1}{6(\frac{n}{2}-2)!\times (4\pi)^{\frac{n}{2}}} \texttt{dim}[S(F)\otimes\wedge(F^{\perp,\star})]$,
$\texttt{dim}[S(F)\otimes\wedge(F^{\perp,\star})]$ equals $2^{\frac{p}{2}+q}$ (resp., $2^{\frac{p-1}{2}+q}$) when $p$ is even (resp., odd).
\end{thm}

\begin{rem}
Let $(M^{n},F)$ be  a compact even-dimensional oriented foliation with spin leave and codimension $q$, and $D_{F}$ be the sub-Dirac operator.
When $p=n$ and $q=0$, then $D_{F}$ is the Dirac operator and we get the classical Kastler-Kalau-Walze theorem for the Dirac operators.
 When $p=0$ and $q=n$, then $D_{F}$ is the de-Rham Hodge operator and we get the classical Kastler-Kalau-Walze theorem for the
 de-Rham Hodge operator.
\end{rem}
Let us now consider the lower dimensional volumes of foliations. The lower dimensional volume of a compact Riemannian
manifold $(M^{n},g)$ without boundary was defined in \cite{RP}.
Let $(M^{n},F)$ be a compact oriented foliation with spin leave and $D_{F}$ be the associated sub-Dirac operator.
Similarly to Proposition 2.3 and Proposition 3.2 in \cite{RP}, the definition of the lower dimensional volumes for foliations is given as follows.
\begin{defn}
The lower dimensional volume of $(M^{n},F)$ is defined by
\begin{equation}
Vol^{(k)}_{(n,p)}(M,F):=Wres(D_{F}^{-k}).
\end{equation}
\end{defn}

\begin{prop}
Let $(M^{n},F)$be a compact foliation without boundary, then

(1) $Vol^{(k)}_{(n,p)}(M,F)$ vanishes when $k$ is odd and $n$ is even, or $k$ is even and $n$ is odd;

(2) when $k$ is even and $n$ is even, we have
\begin{equation}
Vol^{(k)}_{(n,p)}(M,F)=v_{n,k}\int_{M}a_{n-k}\texttt{d}v_{g}(x), \ \ v_{n,k}=\frac{k}{n}(2\pi)^{\frac{k-n}{2}}
\frac{\Gamma(\frac{n}{2}+1)^{\frac{k}{n}}}{\Gamma(\frac{k}{2}+1)};
\end{equation}

(3) when $k$ is odd and $n$ is odd, we have
\begin{equation}
Vol^{(k)}_{(n,p)}(M,F)=v_{n,k}\int_{M}a_{n-k}\texttt{d}v_{g}(x), \ \ v_{n,k}=\frac{k}{n}2^{\frac{(k-n)(n+1)}{2n}}
\pi^{\frac{k-n}{2}}\frac{\Gamma(\frac{n}{2}+1)^{\frac{k}{n}}}{\Gamma(\frac{k}{2}+1)},
\end{equation}
where $a_{n-k}$ is a linear combination of complete contractions of weight $n-k$ of covariant derivatives of the curvature tensor.
The coefficients of this linear combination depend only on $n-k$.
\end{prop}

As a consequence we see that the lower dimensional volumes for foliation are integrals of local Riemannian invariants.
The definition of the lower dimensional volumes for any foliation is defined as follows.

\begin{defn}
Let $(M^{n},F)$ be a compact foliation. Then for $k=1,\ldots,n$, the $k'th$ dimensional volume of $(M^{n},g)$ is:

(1) If $k$ is even and $n$ is even, or $k$ is odd and $n$ is odd,

\begin{equation}
Vol^{(k)}_{(n,p)}(M,F)=v_{n,k}\int_{M}a_{n-k}\texttt{d}v_{g}(x);
\end{equation}

(2) If $k$ is odd and $n$ is even, or $k$ is even and $n$ is odd,
\begin{equation}
Vol^{(k)}_{(n,p)}(M,F)=0.
\end{equation}
where $a_{n-k}$ is the coefficient of $t^{\frac{k-n}{2}}$ in the heat kernel asymptotics for sub-Dirac operator $D_F^2$.
\end{defn}

Now, we compute the lower dimension volumes for 4-dimension foliations. Hence from definition 2.8, we are going to compute the coefficients $a_{n-k}$.
From Theorem 4.1.6 in \cite{Gi}, we obtain the first three coefficients of the heat trace asymptotics
\begin{eqnarray}
a_0(D_F^{2})&=&(4\pi)^{-\frac{n}{2}}\int_M\texttt{tr}\texttt{(Id)}\texttt{d}vol,\\
a_2(D_F^{2})&=&(4\pi)^{-\frac{n}{2}}\int_M\texttt{tr}[(r_M+6E)/6]\texttt{d}vol,\\
a_4(D_F^{2})&=&\frac{(4\pi)^{-\frac{n}{2}}}{360}\int_M\texttt{tr}[-12R_{ijij,kk}+5R_{ijij}R_{klkl}\nonumber\\
&&-2R_{ijik}R_{ljlk}+2R_{ijkl}R_{ijkl}-60R_{ijij}E+180E^2+60E_{,kk}+30\Omega_{ij}
\Omega_{ij}]\texttt{d}vol,
\end{eqnarray}
where
\begin{eqnarray}
-E&=&\frac{r_M}{4}+W \nonumber\\
&=&\frac{r_M}{4}+\frac{1}{4}\sum_{i=1}^{2p}\sum_{r,s,t=1}^q\left<R^{F^\bot}(f_i,h_r)h_t,h_s\right>
c(f_i)c(h_r)\widehat{c}(h_s)\widehat{c}(h_t) \nonumber\\
&&+\frac{1}{8}\sum_{i,j=1}^{2p}\sum_{s,t=1}^q\left<R^{F^\bot}(f_i,f_j)h_t,h_s\right>
c(f_i)c(f_j)\widehat{c}(h_s)\widehat{c}(h_t) \nonumber\\
&&+\frac{1}{8}\sum_{s,t,r,l=1}^q\left<R^{F^\bot}(h_r,h_l)h_t,h_s\right>
c(h_r)c(h_l)\widehat{c}(h_s)\widehat{c}(h_t),
\end{eqnarray}
and
\begin{equation}
\Omega_{ij}=\widetilde{\nabla}_{e_i}\widetilde{\nabla}_{e_j}-\widetilde{\nabla}_{e_j}\widetilde{\nabla}_{e_i}
-\widetilde{\nabla}_{[e_i,e_j]},
\end{equation}
with $e_i$ is $f_i$ or $h_s$.

Since ${\rm dim}[S(F)\otimes\wedge(F^{\perp,\star})]=2^{p+q}$ and
$n=2p+q$, we have
\begin{equation}
a_0(D_F^{2})=\frac{1}{2^p\pi^{p+\frac{q}{2}}}\int_M\texttt{d}vol.
\end{equation}
We note that the trace of the odd degree operator is zero, thus cyclicity of the trace and Clifford relations yield
\begin{eqnarray}
&&\texttt{tr}(c(f_i))=0;~\texttt{tr}(c(f_i)c(f_j))=0 ~{\rm for} ~i\neq j;\nonumber\\
&&\texttt{tr}(c(h_r)c(h_l)\widehat{c}(h_s)\widehat{c}(h_t))=0, ~{\rm for}
~r\neq l.
\end{eqnarray}
Moreover
\begin{equation}
\texttt{tr}E=-2^{p+q}\cdot\frac{r_M}{4},
\end{equation}
and
\begin{equation}
a_2(D_F^{2})=-\frac{1}{12\cdot2^p\pi^{p+\frac{q}{2}}}\int_Mr_M\texttt{d}vol.
\end{equation}
Let $I_1,I_2, I_3$ denote respectively the last three terms in (2.20). Hence
\begin{eqnarray}
\texttt{tr}(E^2)&=&\texttt{tr}(\frac{r_M^2}{16}+W^2)=\texttt{tr}(\frac{r_M^2}{16}+I_1^2+I_2^2+I_3^2),\\
\texttt{tr}(I_1^2)&=&\frac{1}{16}\sum_{i,i'=1}^{2p}\sum_{r,r',s,s',t,t'=1}^q\left<R^{F^\bot}(f_i,h_r)h_t,h_s\right>
   \left<R^{F^\bot}(f_{i'},h_{r'})h_{t'},h_{s'}\right>\nonumber\\
&&\times\texttt{tr}[c(f_i)c(h_r)\widehat{c}(h_s)\widehat{c}(h_t) c(f_{i'})c(h_{r'})\widehat{c}(h_{s'})\widehat{c}(h_{t'})].
\end{eqnarray}
Similar to (2.23), we have
\begin{equation}
\texttt{tr}[c(f_i)c(h_r)\widehat{c}(h_s)\widehat{c}(h_t)c(f_{i'})c(h_{r'})\widehat{c}(h_{s'})\widehat{c}(h_{t'})]
=-\delta_i^{i'}\delta_r^{r'}2^p\texttt{tr}_{\wedge(F^{\perp,\star})}[\widehat{c}(h_s)\widehat{c}(h_t)\widehat{c}(h_{s'})\widehat{c}(h_{t'})].
\end{equation}
Considering $t\neq s,~t'\neq s'$, then
\begin{equation}
\texttt{tr}_{\wedge(F^{\perp,\star})}[\widehat{c}(h_s)\widehat{c}(h_t)\widehat{c}(h_{s'})\widehat{c}(h_{t'})]
=(\delta_t^{s'}\delta_s^{t'}-\delta_t^{t'}\delta_s^{s'})2^q.
\end{equation}
Combining (2.27), (2.28) and (2.29), we obtain
\begin{equation}
\texttt{tr}(I_1^2)=\frac{2^{p+q}}{8}\sum_{i=1}^{2p}\sum_{r,s,t=1}^q\left<R^{F^\bot}(f_i,h_r)h_t,h_s\right>^2.
\end{equation}
Similarly we have
\begin{eqnarray}
&&\texttt{tr}(I_2^2)=\frac{2^{p+q}}{16}\sum_{i,j=1}^{2p}\sum_{s,t=1}^q\left<R^{F^\bot}(f_i,f_j)h_t,h_s\right>^2;\\
&&\texttt{tr}(I_3^2)=\frac{2^{p+q}}{16}
\sum_{s,t,r,l=1}^q\left<R^{F^\bot}(h_r,h_l)h_t,h_s\right>^2.
\end{eqnarray}
Hence in this case
\begin{equation}
\texttt{tr}E^2=\frac{2^{p+q}}{16}r_M^2+\frac{2^{p+q}}{16}||R^{F^\bot}||^2,
\end{equation}
where
\begin{eqnarray}
||R^{F^\bot}||^2&=&2\sum_{i=1}^{2p}\sum_{r,s,t=1}^q\left<R^{F^\bot}(f_i,h_r)h_t,h_s\right>^2
   +\sum_{i,j=1}^{2p}\sum_{s,t=1}^q\left<R^{F^\bot}(f_i,f_j)h_t,h_s\right>^2\nonumber\\
   &&+\sum_{s,t,r,l=1}^q\left<R^{F^\bot}(h_r,h_l)h_t,h_s\right>^2.
\end{eqnarray}

Let us now compute $\texttt{tr}[\Omega_{ij} \Omega_{ij}]$ in a local
coordinate $\{e_{1}, e_{2},e_{3}, e_{4}\}$.
Without loss of generality, we assume $M$ is spin and $\widetilde{\nabla}$ is the standard twisted connection on the
twisted spinors bundle $S(TM)\otimes S(F^\bot)$. A simple computation shows
\begin{eqnarray}
\Omega_{ij}&=&R^{S(TM)}(e_i,e_j)\otimes {\rm Id}_{S(F^\bot)}+{\rm
   Id}_{S(TM)}\otimes R^{S(F^\bot)}(e_i,e_j)\nonumber\\
  &=&-\frac{1}{4}R^M_{ijkl}c(e_k)c(e_l)
  \otimes {\rm Id}_{S(F^\bot)}-\frac{1}{4}{\rm Id}_{S(TM)}\otimes
   \left<R^{F^\bot}(e_i,e_j)h_s,h_t\right>c(h_s)c(h_t).
\end{eqnarray}
Similar to the computations of (2.24), we have
\begin{equation}
\texttt{tr}[\Omega_{ij}\Omega_{ij}]=-\frac{2^{p+q}}{8}(R_{ijkl}^2+||R^{F^\bot}||^2).
\end{equation}
Substituting into (2.21) and by the divergence theorem, we have
\begin{equation}
a_4(D_F^2)=\frac{1}{360\cdot2^p\pi^{p+\frac{q}{2}}}\int_M\left(\frac{5}{4}r_M^2-2R_{ijik}R_{ljlk}-\frac{7}{4}
R_{ijkl}^2+\frac{15}{2}||R^{F^\bot}||^2\right)\texttt{d}vol.
\end{equation}
Then we obtain the lower dimensional volumes of foliations.
\begin{thm}
Let $(M^{n},F)$ be  a compact n-dimensional oriented foliation with spin leave and codimension $q$, and $D_{F}$ be the sub-Dirac operator,
then we have
\begin{eqnarray}
&& Vol^{(n-4)}_{(n,2p)}(M,F)=\frac{v_{n,n-4}}{360\cdot2^p\pi^{p+\frac{q}{2}}}\int_M\left(\frac{5}{4}r_M^2-2R_{ijik}R_{ljlk}-\frac{7}{4}
R_{ijkl}^2+\frac{15}{2}||R^{F^\bot}||^2\right)\texttt{d}vol;  \\
&& Vol^{(n-2)}_{(n,2p)}(M,F)=-\frac{v_{n,n-2}}{12\cdot2^p\pi^{p+\frac{q}{2}}}\int_Mr_M\texttt{d}vol;  \\
&& Vol^{(n)}_{(n,2p)}(M,F)=\frac{v_{n,n}}{2^p\pi^{p+\frac{q}{2}}}\int_M\texttt{d}vol .
\end{eqnarray}
\end{thm}

\begin{rem}
In Theorem 2.9, we assume that $dim F=2p$. When $dim F$ is odd, we can get the similar results.
\end{rem}

\section{A Kastler-Kalau-Walze Type Theorem for foliations with boundary}
In this section, we compute the lower dimension volume for 4-dimensional foliations with boundary and get a Kastler-Kalau-Walze type Formula
in this case.

\subsection{A Kastler-Kalau-Walze Type Theorem for 4-dimensional foliations with boundary}
Let $M$ be a n-dimensional foliation with boundary $\partial M$. We assume that the metric
$g^{M}$ on $M$ has the following form near the boundary
 \begin{equation}
 g^{M}=\frac{1}{h(x_{n})}g^{\partial M}+\texttt{d}x _{n}^{2},
\end{equation}
where $g^{\partial M}$ is the metric on $\partial M$; $h(x_{n})\in C^{\infty}([0,1))=\{g|_{[0,1)}| g\in C^{\infty}((-\varepsilon,1))\}$
for some sufficiently small $\varepsilon>0$ and satisfied $h(x_{n})>0, \ h(0)=1$, where $x_{n}$ denotes the normal directional coordinate.
Let $p_{1},p_{2}$ be nonnegative integers and $p_{1}+p_{2}\leq n$. Then Wang pointed out
the definition of lower volumes for compact connected manifolds with boundary in \cite{Wa4}.
Similarly, we define
\begin{defn} Lower-dimensional volumes of compact connected foliations with boundary are defined by
 \begin{equation}\label{}
   Vol_{(n,p)}^{(p_{1},p_{2})}(M,F):=\widetilde{Wres}[\pi^{+}D_{F}^{-p_{1}} \circ\pi^{+}D_{F}^{-p_{2}}] .
\end{equation}
where $\widetilde{Wres}$ denotes the noncommutative residue for manifolds with boundary in \cite{FGLS}.
\end{defn}

  Denote by $\sigma_{l}(D_{F})$ the \emph{l}-order symbol of an operator $D_{F}$.
Similarly to (2.1.4)-(2.1.8) in \cite{Wa3}, we obtain
 \begin{equation}\label{}
 \widetilde{{\rm Wres}}[\pi^+D_{F}^{-p_1}\circ\pi^+D_{F}^{-p_2}]=\int_M\int_{|\xi|=1}{\rm
trace}_{S(F)\otimes\wedge(F^{\perp,\star})}[\sigma_{-n}(D_{F}^{-p_1-p_2})]\sigma(\xi)dx+\int_{\partial M}\Phi,
\end{equation}
where
 \begin{eqnarray}
\Phi&=&\int_{|\xi'|=1}\int_{-\infty}^{+\infty}\sum_{j,k=0}^{\infty}\sum \frac{(-i)^{|\alpha|+j+k+1}}{\alpha!(j+k+1)!}
    \text{trace}_{S(F)\otimes\wedge(F^{\perp,\star})}[\partial_{x_{n}}^{j}\partial_{\xi'}^{\alpha}\partial_{\xi_{n}}^{k}\sigma_{r}^{+}
    (D_{F}^{-p_{1}})(x',0,\xi',\xi_{n})\nonumber\\
&&\times\partial_{x'}^{\alpha}\partial_{\xi_{n}}^{j+1}\partial_{x_{n}}^{k}\sigma_{l}(D^{-p_{2}})(x',0,\xi',\xi_{n})
\texttt{d}\xi_{n}\sigma(\xi')\texttt{d}x' ,
\end{eqnarray}
and the sum is taken over $r-k+|\alpha|+\ell-j-1=-n, \ r\leq-p_{1}, \ \ell\leq-p_{2}$.

Since $[\sigma_{-n}(D_{F}^{-p_{1}-p_{2}})]|_{M}$ has the same expression with the case of  without boundary in \cite{Wa3}, so locally
we can use Theorem 2.9 to compute the first term. Now let us  give explicit formulas for the volume in dimension 4.
Where $v_{4,2}=\frac{2}{4}(2\pi)^{\frac{2-4}{2}}\frac{\Gamma(\frac{4}{2}+1)^{\frac{2}{4}}}{\Gamma(\frac{2}{2}+1)}
=\frac{1}{2\pi\sqrt{2}}.$
An application of Theorem 2.9 shows that
\begin{equation}
\int_M\int_{|\xi|=1}{\rm trace}_{S(TM)}[\sigma_{-4}(D_{F}^{-1-1})]\sigma(\xi)dx
=-\frac{1}{24\sqrt{2}\cdot2^p\pi^{p+\frac{q}{2}+1}}\int_Mr_M\texttt{d}vol.
\end{equation}
Therefore, we only need to compute $\int_{\partial M}\Phi$.

Let us now turn to compute the symbol expansion of $D_{F}^{-1}$. Recall the definition of the sub-Dirac operator $D_{F}$ in (2.12).
Let $\tilde{\nabla}^{TM}$
denote the Levi-civita connection about $g^M$. In the local coordinates $\{x_i; 1\leq i\leq n\}$ and the fixed orthonormal frame
$\{\widetilde{e_1},\cdots,\widetilde{e_n}\}$, the connection matrix $(\omega_{s,t})$ is defined by
\begin{equation}
\tilde{\nabla}^{TM}(\widetilde{e_1},\cdots,\widetilde{e_n})= (\widetilde{e_1},\cdots,\widetilde{e_n})(\omega_{s,t}).
\end{equation}
Let $c(\widetilde{e_i})$ denote the Clifford action. Let $g^{ij}=g(dx_i,dx_j)$ and
\begin{equation}
\nabla^{TM}_{\partial_i}\partial_j=\sum_k\Gamma_{ij}^k\partial_k; ~\Gamma^k=g^{ij}\Gamma_{ij}^k.
\end{equation}
Let the cotangent vector $\xi=\sum \xi_jdx_j$ and $\xi^j=g^{ij}\xi_i$.
We shall make use of the following convention on the ranges of indices, $i,j,k,l\in F$ and $s,t,r,u \in F^{\perp}$,
and we shall agree that repeated indices are summed over the respective ranges.
From  (2.8) and (2.12), we obtain the sub-Dirac operator
\begin{eqnarray}
D_{F}&=&\sum_{i=1}^{p}c(f_i)\widetilde{\nabla}^{F}_{f_i}+\sum_{s=1}^{q}c(h_s)\widetilde{\nabla}^{F}_{h_s}\nonumber\\
&=&\sum_{i=1}^{p}c(f_i)\Big(\nabla^{S(F)}_{f_i}\otimes \texttt{Id}_{\wedge(F^{\perp,\star})}+\texttt{Id}_{S(F)}\otimes
\nabla^{\wedge(F^{\perp,\star})}_{f_i}+\frac{1}{2}\sum_{i,j=1}^{p}\sum_{s=1}^q\langle S(f_i)f_j,h_s \rangle c(f_j)c(h_s)\Big)\nonumber\\
&&+\sum_{s=1}^{q}c(h_s)\Big(\nabla^{S(F)}_{h_s}\otimes \texttt{Id}_{\wedge(F^{\perp,\star})}+\texttt{Id}_{S(F)}\otimes
\nabla^{\wedge(F^{\perp,\star})}_{h_s}+\frac{1}{2}\sum_{s,t=1}^{q}\sum_{i=1}^{p} \langle S(h_s)h_t,f_j \rangle c(h_t)c(f_j)\Big)\nonumber\\
&=&\Big(\sum_{i}^{p}c(f_{i})f_{i}-\frac{1}{4}\sum_{k,l}\omega_{k,l}(f_{i})c(f_{i})c(f_{k})c(f_{l})\Big)\otimes
{\rm Id}_{\wedge(F^{\perp,\star})}\nonumber\\
&&+\sum_{i}^{p}c(f_{i})\otimes \Big(f_{i}+\frac{1}{4}\sum_{r,t}\omega_{r,t}
(f_{i})[\bar{c}(h_{r})\bar{c}(h_{t})-c(h_{r})c(h_{t})]\Big) \nonumber\\
&&+\Big(h_{s}-\frac{1}{4}\sum_{k,l}\omega_{k,l}(h_{s})c(f_{k})c(f_{l})\Big)\otimes\sum_{s=1}^{q}c(h_{s}) \nonumber\\
&&+{\rm Id}_{S(F)}\otimes \sum_{s=1}^{q}c(h_{s})\Big(h_{s}+\frac{1}{4}\sum_{r,t}\omega_{r,t}
(h_{s})[\bar{c}(h_{r})\bar{c}(h_{t})-c(h_{r})c(h_{t})] \Big)\nonumber\\
&&+\frac{1}{2}\sum_{i,j=1}^{p}\sum_{s=1}^{q}\langle\nabla^{TM}_{f_{i}}f_{j}, h_{s}\rangle c(f_{i})c(f_{j})c(h_{s})
+\frac{1}{2}\sum_{s,t=1}^{q}\sum_{i=1}^{p}\langle\nabla^{TM}_{h_{s}}h_{t}, f_{i}\rangle c(h_{s})c(h_{t})c(f_{i})\nonumber\\
&=&\sum_{i}^{p}c(f_{i})f_{i}+\sum_{s=1}^{q}c(h_{s})h_{s}
-\frac{1}{4}\sum_{i,k,l}\omega_{k,l}(f_{i})c(f_{i})c(f_{k})c(f_{l})\otimes {\rm Id}_{\wedge(F^{\perp,\star})}\nonumber\\
&&-\frac{1}{4}\sum_{s,k,l}\omega_{k,l}(f_{i})c(f_{k})c(f_{l})c(h_{s})\otimes {\rm Id}_{\wedge(F^{\perp,\star})}\nonumber\\
&&+{\rm Id}_{S(F)}\otimes\frac{1}{4}\sum_{i,r,t}\omega_{r,t}(f_{i})c(f_{i})[\bar{c}(h_{r})\bar{c}(h_{t})-c(h_{r})c(h_{t})]\nonumber\\
&&+{\rm Id}_{S(F)}\otimes\frac{1}{4}\sum_{s,r,t}\omega_{r,t}(h_{s})c(h_{s})[\bar{c}(h_{r})\bar{c}(h_{t})-c(h_{r})c(h_{t})]\nonumber\\
&&+\frac{1}{2}\sum_{i,j=1}^{p}\sum_{s=1}^{q}\langle\nabla^{TM}_{f_{i}}f_{j}, h_{s}\rangle c(f_{i})c(f_{j})c(h_{s})
+\frac{1}{2}\sum_{s,t=1}^{q}\sum_{i=1}^{p}\langle\nabla^{TM}_{h_{s}}h_{t}, f_{i}\rangle c(h_{s})c(h_{t})c(f_{i}).
\end{eqnarray}
Then from (3.8), we have
\begin{eqnarray}
\sigma_1(D_{F})&=&\sqrt{-1}c(\xi),\\
\sigma_0(D_{F})&=&
-\frac{1}{4}\sum_{i,k,l}\omega_{k,l}(f_{i})c(f_{i})c(f_{k})c(f_{l})\otimes {\rm Id}_{\wedge(F^{\perp,\star})}\nonumber\\
&&-\frac{1}{4}\sum_{s,k,l}\omega_{k,l}(f_{i})c(f_{k})c(f_{l})c(h_{s})\otimes {\rm Id}_{\wedge(F^{\perp,\star})}\nonumber\\
&&+{\rm Id}_{S(F)}\otimes\frac{1}{4}\sum_{i,r,t}\omega_{r,t}(f_{i})c(f_{i})[\bar{c}(h_{r})\bar{c}(h_{t})-c(h_{r})c(h_{t})]\nonumber\\
&&+{\rm Id}_{S(F)}\otimes\frac{1}{4}\sum_{s,r,t}\omega_{r,t}(h_{s})c(h_{s})[\bar{c}(h_{r})\bar{c}(h_{t})-c(h_{r})c(h_{t})]\nonumber\\
&&+\frac{1}{2}\sum_{i,j=1}^{p}\sum_{s=1}^{q}\langle\nabla^{TM}_{f_{i}}f_{j}, h_{s}\rangle c(f_{i})c(f_{j})c(h_{s})
+\frac{1}{2}\sum_{s,t=1}^{q}\sum_{i=1}^{p}\langle\nabla^{TM}_{h_{s}}h_{t}, f_{i}\rangle c(h_{s})c(h_{t})c(f_{i}).
\end{eqnarray}
 By Lemma 1 in \cite{Wa4} and Lemma 2.1 in \cite{Wa3}, for any fixed point $x_0\in\partial M$, we can choose the normal coordinates $U$
 of $x_0$ in $\partial M$ (not in $M$). By the composition formula and (2.2.11) in \cite{Wa3}, we obtain
\begin{lem}
 Let $D_{F}$ be the sub-Dirac operator associated to $g$ on  $\Gamma({S(F)\otimes\wedge(F^{\perp,\star})})$.
Then
\begin{eqnarray}
&&\sigma_{-1}(D_{F}^{-1})=\frac{\sqrt{-1}c(\xi)}{|\xi|^2}; \\
&& \sigma_{-2}(D_{F}^{-1})=\frac{c(\xi)p_0c(\xi)}{|\xi|^4}+\frac{c(\xi)}{|\xi|^6}\sum_jc(dx_j)
[\partial_{x_j}[c(\xi)]|\xi|^2-c(\xi)\partial_{x_j}(|\xi|^2)],
\end{eqnarray}
where $p_0=\sigma_0(D_{F})$.
\end{lem}

As in \cite{Wa3}, we take normal coordinates in boundary and we get orthonormal frames $\{\tilde{e}_1,\cdots,\tilde{e}_{n-1}\}$.
Note that $\tilde{e}_i$ is not $f_{i}$ or $h_{s}$ in general. We assume $dx_{n}=h_{q}^\star\in\Gamma(F^{\perp,\star})$.
When $dx_{n}$ is not in $\Gamma(F^{\perp,\star})$, we can prove it in a similar way.

Let us now turn to compute $\Phi$ (see formula (3.4) for definition of $\Phi$). Since the sum is taken over $-r-\ell+k+j+|\alpha|=3,
 \ r, \ell\leq-1$, then we have the following five cases:

 {\bf case a)~I)}~$r=-1,~l=-1~k=j=0,~|\alpha|=1$

From (3.4) we have
 \begin{equation}
{\rm case~a)~I)}=-\int_{|\xi'|=1}\int^{+\infty}_{-\infty}\sum_{|\alpha|=1}
{\rm trace} [\partial^\alpha_{\xi'}\pi^+_{\xi_n}\sigma_{-1}(D_{F}^{-1})\times
\partial^\alpha_{x'}\partial_{\xi_n}\sigma_{-1}(D_{F}^{-1})](x_0)d\xi_n\sigma(\xi')dx',
\end{equation}
By Lemma 2.2 in \cite{Wa3}, for $i<n$, then
 \begin{equation}
\partial_{x_i}\sigma_{-1}(D_{F}^{-1})(x_0)=\partial_{x_i}\left(\frac{\sqrt{-1}c(\xi)}{|\xi|^2}\right)(x_0)=
\frac{\sqrt{-1}\partial_{x_i}[c(\xi)](x_0)}{|\xi|^2}
-\frac{\sqrt{-1}c(\xi)\partial_{x_i}(|\xi|^2)(x_0)}{|\xi|^4}=0,
\end{equation}
so \textbf{case a) I)} vanishes.

{\bf case a)~II)}~$r=-1,~l=-1~k=|\alpha|=0,~j=1$

From (3.4) we have
 \begin{equation}
{\rm case \
a)~II)}=-\frac{1}{2}\int_{|\xi'|=1}\int^{+\infty}_{-\infty} {\rm
trace} [\partial_{x_n}\pi^+_{\xi_n}\sigma_{-1}(D_{F}^{-1})\times
\partial_{\xi_n}^2\sigma_{-1}(D_{F}^{-1})](x_0)d\xi_n\sigma(\xi')dx',
\end{equation}
For any fixed point $x_{0}\in\partial M$ and the metric $g^{M}=\frac{1}{H(x_{n})}g^{\partial M}+\texttt{d}x _{n}^{2}$.
Let $\xi=\xi'_{1}+\xi'_{2}+\xi_nc(dx_n)=\xi'+\xi_nc(dx_n)$, where $\xi'_{1}\in\Gamma(F^{\star})$, $\xi'_{2}\in\Gamma(F^{\perp,\star})$.
An application of Lemma 2.1 and Lemma 2.2 in \cite{Wa3} shows that
  \begin{equation}
\partial^2_{\xi_n}\sigma_{-1}(D_{F}^{-1})=\sqrt{-1}\left(-\frac{6\xi_nc(dx_n)+2c(\xi')}
{|\xi|^4}+\frac{8\xi_n^2c(\xi)}{|\xi|^6}\right);
\end{equation}
and
  \begin{equation}
\partial_{x_n}\sigma_{-1}(D_{F}^{-1})(x_0)
=\frac{\sqrt{-1}\partial_{x_n}c(\xi')(x_0)}{|\xi|^2}-\frac{\sqrt{-1}c(\xi)|\xi'|^2h'(0)}{|\xi|^4}.
\end{equation}
By (2.1) in \cite{Wa3} and the Cauchy integral formula, we obtain
\begin{eqnarray}
\pi^+_{\xi_n}\Big[\frac{c(\xi)}{|\xi|^4}\Big](x_0)|_{|\xi'|=1}
&=&\pi^+_{\xi_n}\Big[\frac{c(\xi')+\xi_nc(dx_n)}{(1+\xi_n^2)^2}\Big]\nonumber\\
&=&\frac{1}{2\pi i}\lim_{u\rightarrow 0^-}\int_{\Gamma^+}\frac{\frac{c(\xi')+\eta_nc(dx_n)}{(\eta_n+i)^2(\xi_n+iu-\eta_n)}}
{(\eta_n-i)^2}d\eta_n \nonumber\\
&=&\Big[\frac{c(\xi')+\eta_nc(dx_n)}{(\eta_n+i)^2(\xi_n-\eta_n)}\Big]^{(1)}|_{\eta_n=i}\nonumber\\
&=&-\frac{ic(\xi')}{4(\xi_n-i)}-\frac{c(\xi')+ic(dx_n)}{4(\xi_n-i)^2}
\end{eqnarray}
Similarly,
 \begin{equation}
\pi^+_{\xi_n}\left[\frac{\sqrt{-1}\partial_{x_n}c(\xi')}{|\xi|^2}\right](x_0)|_{|\xi'|=1}
=\frac{\partial_{x_n}[c(\xi')](x_0)}{2(\xi_n-i)}.
\end{equation}
Combining (3.17), (3.18) and (3.19), we have
 \begin{equation}
\pi^+_{\xi_n}\partial_{x_n}\sigma_{-1}(D_{F}^{-1})(x_0)|_{|\xi'|=1}=\frac{\partial_{x_n}[c(\xi')](x_0)}{2(\xi_n-i)}+\sqrt{-1}h'(0)
\left[\frac{ic(\xi')}{4(\xi_n-i)}+\frac{c(\xi')+ic(dx_n)}{4(\xi_n-i)^2}\right].
\end{equation}
By the relation of the Clifford action and ${\rm tr}{AB}={\rm tr }{BA}$, we have the equalities:
\begin{eqnarray}
&&{\rm tr}[c(\xi')c(dx_n)]=0;~~{\rm tr}[c(dx_n)^2]=-8;~~{\rm tr}[c(\xi')^2](x_0)|_{|\xi'|=1}=-8;\nonumber\\
&&{\rm tr}[\partial_{x_n}c(\xi')c(dx_n)]=0;~~{\rm tr}[\partial_{x_n}c(\xi')c(\xi')](x_0)|_{|\xi'|=1}=-4h'(0).
\end{eqnarray}
From equation (3.16), (3.20) and (3.21), one sees that
\begin{eqnarray}
&&h'(0){\rm tr}\left\{\left[\frac{ic(\xi')}{4(\xi_n-i)}+\frac{c(\xi')+ic(dx_n)}{4(\xi_n-i)^2}\right]\times
\left[\frac{6\xi_nc(dx_n)+2c(\xi')}{(1+\xi_n^2)^2}-\frac{8\xi_n^2[c(\xi')+\xi_nc(dx_n)]}{(1+\xi_n^2)^3}\right]
\right\}(x_0)|_{|\xi'|=1}\nonumber\\
&=&-8h'(0)\frac{-2i\xi_n^2-\xi_n+i}{(\xi_n-i)^4(\xi_n+i)^3}.
\end{eqnarray}
Similarly, we have
\begin{eqnarray}
&&-\sqrt{-1}{\rm tr}\left\{\left[\frac{\partial_{x_n}[c(\xi')](x_0)}{2(\xi_n-i)}\right]
\times\left[\frac{6\xi_nc(dx_n)+2c(\xi')}{(1+\xi_n^2)^2}-\frac{8\xi_n^2[c(\xi')+\xi_nc(dx_n)]}
{(1+\xi_n^2)^3}\right]\right\}(x_0)|_{|\xi'|=1}\nonumber\\
&=&-4\sqrt{-1}h'(0)\frac{3\xi_n^2-1}{(\xi_n-i)^4(\xi_n+i)^3}.
\end{eqnarray}
Combining (3.22) and (3.23), we obtain
\begin{eqnarray}
{\rm case~ a)~II)}&=&-\int_{|\xi'|=1}\int^{+\infty}_{-\infty}\frac{2ih'(0)(\xi_n-i)^2}
{(\xi_n-i)^4(\xi_n+i)^3}d\xi_n\sigma(\xi')dx'\nonumber\\
&=&-2ih'(0)\Omega_3\int_{\Gamma^+}\frac{1}{(\xi_n-i)^2(\xi_n+i)^3}d\xi_ndx'\nonumber\\
&=&-2ih'(0)\Omega_32\pi i[\frac{1}{(\xi_n+i)^3}]^{(1)}|_{\xi_n=i}dx'\nonumber\\
&=&-\frac{3}{4}\pi h'(0)\Omega_3dx'.
\end{eqnarray}

 {\bf case a)~III)}~$r=-1,~l=-1~j=|\alpha|=0,~k=1$

From (3.4) we have
 \begin{equation}
{\rm case~ a)~III)}=-\frac{1}{2}\int_{|\xi'|=1}\int^{+\infty}_{-\infty}
{\rm trace} [\partial_{\xi_n}\pi^+_{\xi_n}\sigma_{-1}(D_{F}^{-1})\times
\partial_{\xi_n}\partial_{x_n}\sigma_{-1}(D_{F}^{-1})](x_0)d\xi_n\sigma(\xi')dx',
\end{equation}
Then an application of Lemma 2.2 in \cite{Wa3} shows
  \begin{equation}
\partial_{\xi_n}\partial_{x_n}q_{-1}(x_0)|_{|\xi'|=1}=-\sqrt{-1}h'(0)
\left[\frac{c(dx_n)}{|\xi|^4}-4\xi_n\frac{c(\xi')+\xi_nc(dx_n)}{|\xi|^6}\right]-
\frac{2\xi_n\sqrt{-1}\partial_{x_n}c(\xi')(x_0)}{|\xi|^4},
\end{equation}
and
  \begin{equation}
\partial_{\xi_n}\pi^+_{\xi_n}q_{-1}(x_0)|_{|\xi'|=1}=-\frac{c(\xi')+ic(dx_n)}{2(\xi_n-i)^2}.
\end{equation}
Similarly to (3.22) and (3.23), we have
\begin{eqnarray}
&&{\rm tr}\left\{\frac{c(\xi')+ic(dx_n)}{2(\xi_n-i)^2}\times
\sqrt{-1}h'(0)\left[\frac{c(dx_n)}{|\xi|^4}-4\xi_n\frac{c(\xi')+\xi_nc(dx_n)}{|\xi|^6}\right]\right\}\nonumber\\
&=&4h'(0)\frac{i-3\xi_n}{(\xi_n-i)^4(\xi_n+i)^3};
\end{eqnarray}
and
 \begin{equation}
{\rm tr}\left[\frac{c(\xi')+ic(dx_n)}{2(\xi_n-i)^2}\times
\frac{2\xi_n\sqrt{-1}\partial_{x_n}c(\xi')(x_0)}{|\xi|^4}\right]
=-4h'(0)\sqrt{-1}\frac{\xi_n}{(\xi_n-i)^4(\xi_n+i)^2}.
\end{equation}
Therefore, \textbf{case a) III)}$=\frac{3}{4}\pi h'(0)\Omega_3dx'$.

 {\bf case b)}~$r=-2,~l=-1,~k=j=|\alpha|=0$

From (3.4) we have
 \begin{equation}
{\rm case~ b)}=-i\int_{|\xi'|=1}\int^{+\infty}_{-\infty}
{\rm trace} [\pi^+_{\xi_n}\sigma_{-2}(D_{F}^{-1})\times
\partial_{\xi_n}\sigma_{-1}(D_{F}^{-1})](x_0)d\xi_n\sigma(\xi')dx',
\end{equation}
 By Lemma 2.1 and Lemma 2.2 in \cite{Wa3},  we obtain
\begin{equation}
\partial_{\xi_n}\sigma_{-1}(D_{F}^{-1})(x_0)|_{|\xi'|=1}=\sqrt{-1}
\left[\frac{c(dx_n)}{1+\xi_n^2}-\frac{2\xi_nc(\xi')+2\xi_n^2c(dx_n)}{(1+\xi_n^2)^2}\right]
=\frac{(i-i\xi_n^2)c(dx_n)-2i\xi_nc(\xi')}{(1+\xi_n^2)^2},
\end{equation}
and
  \begin{equation}
\sigma_{-2}(D_{F}^{-1})(x_0)=\frac{c(\xi)p_0(x_0)c(\xi)}{|\xi|^4}+\frac{c(\xi)}{|\xi|^6}c(dx_n)
[\partial_{x_n}[c(\xi')](x_0)|\xi|^2-c(\xi)h'(0)|\xi|^2_{\partial
M}].
\end{equation}
From (3.32), one sees that
 \begin{eqnarray}
&&\pi^+_{\xi_n}\sigma_{-2}(D_{F}^{-1})(x_0)|_{|\xi'|=1}\nonumber\\
&=&\pi^+_{\xi_n}\left[\frac{c(\xi)p_0(x_0)c(\xi)+c(\xi)c(dx_n)\partial_{x_n}[c(\xi')](x_0)}{(1+\xi_n^2)^2}\right]
-h'(0)\pi^+_{\xi_n}\left[\frac{c(\xi)c(dx_n)c(\xi)}{(1+\xi_n^2)^3}\right]\nonumber\\
&:=&B_1-B_2,
\end{eqnarray}
where
\begin{eqnarray}
B_2&=&h'(0)\pi^+_{\xi_n}\left[\frac{c(\xi)c(dx_n)c(\xi)}{(1+\xi_n^2)^3}\right]\nonumber\\
&=&h'(0)\pi_{\xi_n}^+\left[\frac{-\xi_n^2c(dx_n)^2-2\xi_nc(\xi')+c(dx_n)}{(1+\xi_n^2)^3}\right] \nonumber\\
&=&\frac{h'(0)}{2}\left[\frac{-\eta^2_nc(dx_n)-2\eta_nc(\xi')+c(dx_n)}{(\eta_n+i)^3(\xi_n-\eta_n)}\right]^{(2)}|_{\eta_n=i}\nonumber\\
&=&\frac{h'(0)}{2}\left[\frac{c(dx_n)}{4i(\xi_n-i)}+\frac{c(dx_n)-ic(\xi')}{8(\xi_n-i)^2}
+\frac{3\xi_n-7i}{8(\xi_n-i)^3}[ic(\xi')-c(dx_n)]\right].
\end{eqnarray}
 By (3.31) and (3.34), we have
\begin{equation}
{\rm tr }[B_2\times\partial_{\xi_n}\sigma_{-1}(D_{F}^{-1})(x_0)]|_{|\xi'|=1}
=\frac{\sqrt{-1}}{2}h'(0)\frac{-i\xi_n^2-\xi_n+4i}{4(\xi_n-i)^3(\xi_n+i)^2}
{\rm tr}_{({S(F)\otimes\wedge(F^{\perp,\star})})}[{\rm id}],
\end{equation}
where ${\rm tr}_{({S(F)\otimes\wedge(F^{\perp,\star})})}[{\rm id}]=8$.
Hence
\begin{equation}
{\rm tr }[B_2\times\partial_{\xi_n}\sigma_{-1}(D_{F}^{-1})(x_0)]|_{|\xi'|=1}
=h'(0)\frac{\xi_n^2-i\xi_n-4}{(\xi_n-i)^3(\xi_n+i)^2}.
\end{equation}
Similarly to (3.18), we have
\begin{eqnarray}
B_1&=&\frac{-1}{4(\xi_n-i)^2}\Big[(2+i\xi_n)c(\xi')p_0c(\xi')+i\xi_nc(dx_n)p_0c(dx_n)\nonumber\\
&&+(2+i\xi_n)c(\xi')c(dx_n)\partial_{x_n}c(\xi')+ic(dx_n)p_0c(\xi')+ic(\xi')p_0c(dx_n)-i\partial_{x_n}c(\xi')\Big]\nonumber\\
&:=&C_1+C_2,
\end{eqnarray}
where
\begin{equation}
C_1:=\frac{-1}{4(\xi_n-i)^2}\big[(2+i\xi_n)c(\xi')p_0c(\xi')+i\xi_nc(dx_n)p_0c(dx_n)+ic(dx_n)p_0c(\xi')+ic(\xi')p_0c(dx_n)\big],
\end{equation}
\begin{equation}
C_2:=\frac{-1}{4(\xi_n-i)^2}\big[(2+i\xi_n)c(\xi')c(dx_n)\partial_{x_n}c(\xi')-i\partial_{x_n}c(\xi')\big].
\end{equation}
Combining (3.31) and (3.39), we have
\begin{equation}
{\rm tr }[C_2\times\partial_{\xi_n}\sigma_{-1}(D_{F}^{-1})(x_0)]|_{|\xi'|=1}
=h'(0)\frac{\xi_n^2-i\xi_n-2}{(\xi_n-i)^3(\xi_n+i)^2}.
\end{equation}

On the other hand, let $c(\xi')=\sum_{j=1}^{p}a_{j}c(f_{j})+\sum_{u=1}^{q}b_{u}c(h_{u})~(a_{j}^{2}+b_{u}^{2}=1), c(dx_n)=c(h_{q})$.
By the trace identity  $\texttt{tr}(AB)=\texttt{tr}(BA)$, $\texttt{tr}(A\otimes B)=\texttt{tr}(A)\cdot\texttt{tr}(B)$ and the
relation of the Clifford action, we obtain
\begin{equation}
\texttt{tr}[c(f_i)c(f_j)c(h_s)c(h_q)]=\texttt{tr}(c(f_i)c(f_j))\cdot\texttt{tr}(c(h_s)c(h_q))=\delta_i^{j}\delta_s^{q}2^{p+q},
\end{equation}
and
\begin{equation}
\texttt{tr}\big[c(h_s)[\bar{c}(h_r)\bar{c}(h_t)-c(h_r)c(h_t)]c(h_q)\big]=-\texttt{tr}[c(h_s)c(h_r)c(h_t)c(h_q)]
=-(\delta_r^{s}\delta_t^{q}-\delta_r^{q}\delta_s^{t})2^{q}.
\end{equation}
Similarly,
\begin{eqnarray}
&&\texttt{tr}[c(f_i)c(f_k)c(f_l)c(f_j)]=(\delta_i^{k}\delta_l^{j}-\delta_i^{l}\delta_k^{j})2^{p}; \nonumber\\
&&\texttt{tr}[c(h_s)c(h_t)c(f_i)c(f_j)]=\delta_s^{t}\delta_i^{j}2^{p+q}; \nonumber\\
&&\texttt{tr}\big[c(h_s)[\bar{c}(h_r)\bar{c}(h_t)-c(h_r)c(h_t)]c(h_u)\big]=-(\delta_r^{s}\delta_t^{u}-\delta_r^{u}\delta_s^{t})2^{q};\nonumber\\
&&\texttt{tr}[c(f_i)c(f_j)c(h_s)c(h_u)]=\delta_i^{j}\delta_s^{u}2^{p+q},
\end{eqnarray}
the other is zero.

Combining (3.31), (3.38),  and (3.41)-(3.43),  we have
\begin{eqnarray}
&&{\rm tr }[C_1\times\partial_{\xi_n}\sigma_{-1}(D_{F}^{-1})(x_0)]|_{|\xi'|=1}\nonumber\\
&=&\frac{-1}{4(\xi_n-i)^2(1+\xi_n^2)^2}
    {\rm tr}\Big\{\Big[(2+i\xi_n)c(\xi')p_0c(\xi')+i\xi_nc(dx_n)p_0c(dx_n)\nonumber\\
&&+ic(dx_n)p_0c(\xi')+ic(\xi')p_0c(dx_n)\Big]
[(i-i\xi_n^2)c(dx_n)-2i\xi_nc(\xi')]\Big\}(x_0)|_{|\xi'|=1}\nonumber\\
&=&\frac{-1}{4(\xi_n-i)^2(1+\xi_n^2)^2}
    \Big[(-2i\xi_n^{2}-4\xi_n+2i){\rm tr}[p_0c(dx_n)]+(-2\xi_n^{2}+4i\xi_n+2){\rm tr}[p_0c(\xi')]\Big](x_0)|_{|\xi'|=1}  \nonumber\\
&=&\frac{i}{2(\xi_n-i)^2(\xi_n+i)^2}\Big[-(\frac{1}{4}\sum_{s,t}\omega_{s,t}(h_{s})(\delta_r^{s}\delta_t^{q}-\delta_r^{q}\delta_s^{t}))
+\frac{1}{2}\sum_{i=1}^{p}\langle\nabla^{TM}_{f_{i}}f_{i}, h_{q}\rangle\Big]
\times{\rm tr}_{({S(F)\widehat{\otimes}\wedge(F^{\perp,\star})})}[{\rm id}](x_0)|_{|\xi'|=1}\nonumber\\
&&+\frac{1}{2(\xi_n-i)^2(\xi_n+i)^2}
\Big[\sum_{j=1}^{p}a_{j}\Big(-\frac{1}{4}\sum_{i,j}\omega_{i,j}(f_{i})(\delta_i^{k}\delta_l^{j}-\delta_i^{l}\delta_k^{j}))
+\frac{1}{2}\sum_{i=1}^{p}\sum_{s,t=1}^{q}\langle\nabla^{TM}_{h_{s}}h_{t}, f_{i}\rangle \delta_s^{t}\delta_i^{j} \Big)\nonumber\\
&&+\sum_{u=1}^{q}b_{u}\Big(-\frac{1}{4}\sum_{s,u}\omega_{s,u}(h_{s})(\delta_r^{s}\delta_t^{u}-\delta_r^{u}\delta_s^{t}))
+\frac{1}{2}\sum_{i=1}^{p}\sum_{s=1}^{q}\langle\nabla^{TM}_{f_{i}}f_{i}, h_{s}\rangle\delta_i^{j}\delta_s^{u} \Big)\Big]
\times{\rm tr}_{({S(F)\widehat{\otimes}\wedge(F^{\perp,\star})})}[{\rm id}](x_0)|_{|\xi'|=1}\nonumber\\
&=&\frac{2i}{(\xi_n-i)^2(\xi_n+i)^2}\Big[\sum_{s,q}\omega_{q,s}(h_{s})+
\sum_{i=1}^{p}\langle\nabla^{TM}_{f_{i}}f_{i}, h_{q}\rangle\Big](x_0)|_{|\xi'|=1}
+\frac{2}{(\xi_n-i)^2(\xi_n+i)^2}\Big[\sum_{j=1}^{p}a_{j}\times\nonumber\\
&&\Big(-\sum_{i,j}\omega_{i,j}(f_{i})+
\sum_{i=1}^{p}\sum_{s=1}^{q}\langle\nabla^{TM}_{h_{s}}h_{s}, f_{j}\rangle \Big)
+\sum_{u=1}^{q}b_{u}\Big(-\sum_{s,u}\omega_{s,u}(h_{s})+
\sum_{i=1}^{p}\sum_{s=1}^{q}\langle\nabla^{TM}_{f_{i}}f_{i}, h_{u}\rangle \Big)\Big](x_0)|_{|\xi'|=1}\nonumber\\
&=&\frac{2i}{(\xi_n-i)^2(\xi_n+i)^2}\Big[\sum_{s=1}^{q}\langle\nabla^{TM}_{h_{s}}s_{s}, h_{q}\rangle+
\sum_{i=1}^{p}\langle\nabla^{TM}_{f_{i}}f_{i}, h_{q}\rangle\Big](x_0)|_{|\xi'|=1}
+\frac{2}{(\xi_n-i)^2(\xi_n+i)^2}\Big[\Big(\langle\nabla^{TM}_{f_{i}}f_{j}, \sum_{j=1}^{p}a_{j}f_{j}\rangle\nonumber\\
&&+\langle\nabla^{TM}_{f_{i}}f_{j}, \sum_{u=1}^{q}b_{u}h_{u}\rangle+\langle\nabla^{TM}_{h_{s}}h_{s}, \sum_{j=1}^{p}a_{j}f_{j}\rangle
+\langle\nabla^{TM}_{h_{s}}h_{s}, \sum_{u=1}^{q}b_{u}h_{u}\rangle\Big](x_0)|_{|\xi'|=1}\nonumber\\
&=&\frac{2i}{(\xi_n-i)^2(\xi_n+i)^2}\Big(\sum_{i=1}^{p+q}\langle\nabla^{TM}_{\tilde{e}_{i}}\tilde{e}_{i}, dx_{n}\rangle\Big)(x_0)|_{|\xi'|=1}
+\frac{2}{(\xi_n-i)^2(\xi_n+i)^2}\Big(\sum_{i=1}^{p+q}\langle\nabla^{TM}_{\tilde{e}_{i}}\tilde{e}_{i}, \xi'\rangle\Big)(x_0)|_{|\xi'|=1}
\nonumber\\
&=&0,
\end{eqnarray}
where we have used the fact that in normal coordinates, $\sum_{i=1}^{p+q}
\langle\nabla^{TM}_{\tilde{e}_{i}}\tilde{e}_{i}, \xi_{i}\rangle(x_0)=\Big(\sum_{i=1}^{p+q}\langle\nabla^{TM}_{f_{i}}f_{j}, \xi_{i}\rangle
+\sum_{i=1}^{p+q}\langle\nabla^{TM}_{h_{s}}s_{s}, \xi_{i}\rangle\Big)(x_0)=0.$

Combining (3.35), (3.40) and (3.44), we obtain
\begin{eqnarray}
{\rm case~ b)}
&=&-i\int_{|\xi'|=1}\int^{+\infty}_{-\infty}{\rm trace} [\pi^+_{\xi_n}\sigma_{-2}(D_{F}^{-1})\times
\partial_{\xi_n}\sigma_{-1}(D_{F}^{-1})](x_0)d\xi_n\sigma(\xi')dx'\nonumber\\
&=&-i\int_{|\xi'|=1}\int^{+\infty}_{-\infty}{\rm trace} [(C_1+C_2-B_2)\times
\partial_{\xi_n}\sigma_{-1}(D_{F}^{-1})](x_0)d\xi_n\sigma(\xi')dx'\nonumber\\
&=&-i\int_{|\xi'|=1}\int^{+\infty}_{-\infty}\Big[
+h'(0)\frac{\xi_n^2-i\xi_n-2}{(\xi_n-i)^3(\xi_n+i)^2}-h'(0)\frac{\xi_n^2-i\xi_n-4}{(\xi_n-i)^3(\xi_n+i)^2}\Big](x_0)d\xi_n\sigma(\xi')dx'
\nonumber\\
&=&\int_{|\xi'|=1}\int^{+\infty}_{-\infty}\Big[\frac{-2ih'(0)}{(\xi_n-i)^3(\xi_n+i)^2}
\Big](x_0)d\xi_n\sigma(\xi')dx'\nonumber\\
&=&\frac{3}{4} h'(0)\pi\Omega_3dx'.
\end{eqnarray}

{\bf  case c)}~$r=-1,~l=-2,~k=j=|\alpha|=0$

From (3.4) we have
 \begin{equation}
{\rm case~ c)}=-i\int_{|\xi'|=1}\int^{+\infty}_{-\infty}{\rm trace} [\pi^+_{\xi_n}\sigma_{-1}(D_{F}^{-1})\times
\partial_{\xi_n}\sigma_{-2}(D_{F}^{-1})](x_0)d\xi_n\sigma(\xi')dx'.
\end{equation}
From (3.11) and (3.12), we have
\begin{equation}
\pi^+_{\xi_n}\sigma_{-1}(D_{F}^{-1})(x_0)|_{|\xi'|=1}=\frac{c(\xi')+ic(dx_n)}{2(\xi_n-i)},
\end{equation}
and
\begin{eqnarray}
&&\partial_{\xi_n}\sigma_{-2}(D_{F}^{-1})(x_0)|_{|\xi'|=1}\nonumber\\
&=&\frac{1}{(1+\xi_n^2)^3}\Big[(2\xi_n-2\xi_n^3)c(dx_n)p_0c(dx_n)+(1-3\xi_n^2)c(dx_n)p_0c(\xi')\nonumber\\
&&+(1-3\xi_n^2)c(\xi')p_0c(dx_n)-4\xi_nc(\xi')p_0c(\xi')
+(3\xi_n^2-1)\partial_{x_n}c(\xi')-4\xi_nc(\xi')c(dx_n)\partial_{x_n}c(\xi')\nonumber\\
&&+2h'(0)c(\xi')+2h'(0)\xi_nc(dx_n)\Big]
+6\xi_nh'(0)\frac{c(\xi)c(dx_n)c(\xi)}{(1+\xi^2_n)^4}.
\end{eqnarray}
Similarly to (3.44), we obtain
\begin{equation}
{\rm trace}[\pi^+_{\xi_n}\sigma_{-1}(D_{F}^{-1})\times\partial_{\xi_n}\sigma_{-2}(D_{F}^{-1})](x_0)|_{|\xi'|=1}
=h'(0)\frac{-6}{(\xi_n-i)^3(\xi_n+i)^2}+h'(0)\frac{24i\xi_n}{(\xi_n-i)^3(\xi_n+i)^4}.
\end{equation}
Hence
\begin{eqnarray}
{\rm case~ c)}
&=&-i\int_{|\xi'|=1}\int^{+\infty}_{-\infty}{\rm trace} [\pi^+_{\xi_n}\sigma_{-1}(D_{F}^{-1})\times
\partial_{\xi_n}\sigma_{-2}(D_{F}^{-1})](x_0)d\xi_n\sigma(\xi')dx'\nonumber\\
&=&-i\int_{|\xi'|=1}\int^{+\infty}_{-\infty}\Big[
h'(0)\frac{-6}{(\xi_n-i)^3(\xi_n+i)^2}+h'(0)\frac{24i\xi_n}{(\xi_n-i)^3(\xi_n+i)^4}
\Big](x_0)d\xi_n\sigma(\xi')dx'\nonumber\\
&=&\int_{|\xi'|=1}\int^{+\infty}_{-\infty}\Big[h'(0)\frac{6i}{(\xi_n-i)^3(\xi_n+i)^2}
\Big](x_0)d\xi_n\sigma(\xi')dx'\nonumber\\
&&+\int_{|\xi'|=1}\int^{+\infty}_{-\infty}\Big[h'(0)\frac{24\xi_n}{(\xi_n-i)^3(\xi_n+i)^4}
\Big](x_0)d\xi_n\sigma(\xi')dx'\nonumber\\
&=&-\frac{3}{4} h'(0)\pi\Omega_3dx'
\end{eqnarray}

Since $\Phi$  is the sum of the \textbf{case a, b} and \textbf{c}, so is zero. Then we have
\begin{thm}\label{th:53}
Let M be a 4-dimensional compact connected foliation with the boundary $\partial M$ and the metric $g^{M}$ as above ,
and $D_{F}$ be the sub-Dirac operator on  $\Gamma({S(F)\otimes\wedge(F^{\perp,\star})})$, then
\begin{equation}
Vol_{4}^{(1,1)}(M)=-\frac{1}{24\sqrt{2}\cdot2^p\pi^{p+\frac{q}{2}+1}}\int_Mr_M\texttt{d}vol.
\end{equation}
where $r_M$ be the scaler curvature of the foliation.
\end{thm}

\begin{rem}
Let M be a 4-dimensional compact connected foliation with the boundary $\partial M$ and the metric $g^{M}$ as above,
and $D_{F}$ be the sub-Dirac operator. When $n=p$, we get Theorem 2.5 in \cite{Wa3}. When $n=q$, we get Theorem 3.1 in \cite{Wa3}.
\end{rem}

\subsection{The gravitational action for $4$-dimensional foliation with boundary}

 Firstly, we recall the Einstein-Hilbert action for manifolds with boundary in \cite{Wa3},
 \begin{equation}
I_{\rm Gr}=\frac{1}{16\pi}\int_Ms{\rm dvol}_M+2\int_{\partial M}K{\rm dvol}_{\partial_M}:=I_{\rm {Gr,i}}+I_{\rm {Gr,b}},
\end{equation}
 where
 \begin{equation}
K=\sum_{1\leq i,j\leq {n-1}}K_{i,j}g_{\partial M}^{i,j};~~K_{i,j}=-\Gamma^n_{i,j},
\end{equation}
 and $K_{i,j}$ is the second fundamental form, or extrinsic
curvature. Take the metric in Section 2, then by Lemma A.2 in \cite{Wa3},
$K_{i,j}(x_0)=-\Gamma^n_{i,j}(x_0)=-\frac{1}{2}h'(0),$ when $i=j<n$,
otherwise is zero. For $n=4$, we obtain
\begin{equation}
K(x_0)=\sum_{i,j}K_{i.j}(x_0)g_{\partial M}^{i,j}(x_0)=\sum_{i=1}^3K_{i,i}(x_0)=-\frac{3}{2}h'(0).
\end{equation}
 So
 \begin{equation}
I_{\rm {Gr,b}}=-3h'(0){\rm Vol}_{\partial M}.
\end{equation}
 Let $M$ be $4$-dimensional foliation with boundary and $P,P'$ be two pseudodifferential operators with transmission
property (see \cite{Wa1}) on $\widehat M$. From (4.4) in \cite{Wa3},  we have
\begin{equation}
\pi^+P\circ\pi^+P'=\pi^+(PP')+L(P,P')
\end{equation}
and $L(P,P')$ is leftover term which represents the difference between
the composition $\pi^+P\circ\pi^+P'$ in Boutet de Monvel algebra and
the composition $PP'$ in the classical pseudodifferential operators
algebra. By (3.4), we define locally
\begin{eqnarray}
&&{\rm res}_{1,1}(P,P'):=-\frac{1}{2}\int_{|\xi'|=1}\int^{+\infty}_{-\infty}
{\rm trace} [\partial_{x_n}\pi^+_{\xi_n}\sigma_{-1}(P)\times
\partial_{\xi_n}^2\sigma_{-1}(P')]d\xi_n\sigma(\xi')dx'; \\
&&{\rm res}_{2,1}(P,P'):=-i\int_{|\xi'|=1}\int^{+\infty}_{-\infty}
{\rm trace} [\pi^+_{\xi_n}\sigma_{-2}(P)\times
\partial_{\xi_n}\sigma_{-1}(P')]d\xi_n\sigma(\xi')dx'.
\end{eqnarray}

Hence, they represent the difference between the composition
$\pi^+P\circ\pi^+P'$ in Boutet de Monvel algebra and the composition
$PP'$ in the classical pseudodifferential operators algebra
partially. Then
\begin{equation}
{\rm case~ a)~ II)}={\rm res}_{1,1}(D_{F}^{-1},D_{F}^{-1});~{\rm case~ b)}={\rm res}_{2,1}(D_{F}^{-1},D_{F}^{-1}).
\end{equation}

 Now, we assume $\partial M$ is flat , then
$\{dx_i=e_i\},~g^{\partial M}_{i,j}=\delta_{i,j},~\partial
_{x_s}g^{\partial M}_{i,j}=0$. So ${\rm res}_{1,1}(D_{F}^{-1},D_{F}^{-1})$
and ${\rm res}_{2,1}(D_{F}^{-1},D_{F}^{-1})$ are two global forms locally
defined by the aboved oriented orthonormal basis $\{dx_i\}$. From case
a) II) and case b),  we have:

\begin{thm}
Let M be a 4-dimensional flat compact connected foliation with the boundary $\partial M$ and the metric $g^{M}$ as above ,
and $D_{F}$ the sub-Dirac operator on  $\Gamma({S(F)\otimes\wedge(F^{\perp,\star})})$, then
\begin{eqnarray}
&& \int_{\partial M}{\rm res}_{1,1}(D_{F}^{-1},D_{F}^{-1})=\frac{\pi}{4}\Omega_3I_{\rm {Gr,b}};  \\
&&\int_{\partial M}{\rm res}_{2,1}(D_{F}^{-1},D_{F}^{-1})=-\frac{1}{4}\pi\Omega_3I_{\rm {Gr,b}}.
\end{eqnarray}
\end{thm}

\subsection{Computations of  $ \widetilde{{\rm
Wres}}[(\pi^+\widehat{D_{F}}^{-1})^2]$ for $3$-dimensional Foliations}

For an odd dimensional manifolds with boundary, as in Theorem 2.9 and (3.3), we have the formula
\begin{equation}
\widetilde{{\rm Wres}}[(\pi^+D_{F}^{-1})^2]=\int_{\partial M}\Phi.
\end{equation}
From (3.4), when $n=3$,  $ r-k-|\alpha|+l-j-1=-3,~~r,l\leq-1$, so we get $r=l=-1,~k=|\alpha|=j=0,$ then
\begin{equation}
\Phi=\int_{|\xi'|=1}\int^{+\infty}_{-\infty} {\rm trace}_{S(TM)}[ \sigma^+_{-1}(D_{F}^{-1})(x',0,\xi',\xi_n)
\times\partial_{\xi_n}\sigma_{-1}(D_{F}^{-1})(x',0,\xi',\xi_n)]d\xi_3\sigma(\xi')dx'.
\end{equation}
Similar to (3.20), by Lemma 3.2, we have
\begin{equation}
\sigma^+_{-1}(D^{-1})|_{|\xi'|=1}=\frac{\sqrt{-1}[c(\xi')+ic(dx_n)]}{2i(\xi_n-i)};
\end{equation}
and
\begin{equation}
\partial_{\xi_n}\sigma_{-1}(D^{-1})|_{|\xi'|=1}=\frac{\sqrt{-1}c(dx_n)}{1+\xi_n^2}
-\frac{2\sqrt{-1}\xi_nc(\xi)}{(1+\xi_n^2)^2}.
\end{equation}
We take the coordinates as in Section 3. Locally $S(TM)|_{\widetilde {U}}\cong
\widetilde {U}\times\wedge^{{\rm even}} _{\bf C}(2).$
 Let $\{\widetilde{f_1},\widetilde{f_2}\}$ be an orthonormal basis of $\wedge^{{\rm even}} _{\bf C}(2)$ and we will compute the trace
under this basis. Similarly to (3.21), we have
\begin{equation}
{\rm tr}[c(\xi')c(dx_n)]=0;~~{\rm tr}[c(dx_n)^2]=-4;~~{\rm tr}[c(\xi')^2](x_0)|_{|\xi'|=1}=-4.
\end{equation}
Combining (3.3) (3.4) and (3.5), we have
\begin{equation}
{\rm trace} [
\sigma^+_{-1}(D^{-1})\times\partial_{\xi_n}\sigma_{-2}(D^{-1})](x_0)|_{|\xi'|=1}=-\frac{1}{(\xi_n+i)^2(\xi_n-i)}.
\end{equation}
By (4.3), (4.6) and the Cauchy integral formula, we obtain
\begin{equation}
\Phi=i\pi\Omega_2{\rm vol}_{\partial M}=2i\pi^2{\rm vol}_{\partial M},
\end{equation}
where ${\rm vol}_{\partial M}$ denotes the canonical volume form of ${\partial M}$.

\begin{thm}
Let M be a 3-dimensional compact connected foliation with the boundary $\partial M$ and the metric $g^{M}$ as above ,
and $D_{F}$ the sub-Dirac operator on  $\Gamma({S(F)\otimes\wedge(F^{\perp,\star})})$, then
\begin{equation}
\widetilde{{\rm Wres}}[(\pi^+D^{-1})^2]=2i\pi^2{\rm Vol}_{\partial M},
\end{equation}
where ${\rm Vol}_{\partial M}$ denotes the canonical volume of ${\partial M}.$
\end{thm}

\section{A Kastler-Kalau-Walze Type Theorem for 6-dimensional foliations with boundary}

 In this section, We compute the lower dimensional volume ${\rm Vol}^{(2,2)}_6$ for $6$-dimensional foliations with
boundary and get a Kastler-Kalau-Walze type theorem in this case.

\subsection{A Kastler-Kalau-Walze Type Theorem for 6-dimensional foliations with boundary}

Since $[\sigma_{-n}(D_{F}^{-p_{1}-p_{2}})]|_{M}$ has the same expression with the case of  without boundary in \cite{Wa3},
so locally we can use Theorem 2.4 to compute the first term. Now let us  give explicit formulas for the volume in dimension 6.
 Let $dim(S(F))=\tilde{l}$, then
 \begin{equation}
\int_M\int_{|\xi|=1}{\rm trace}_{S(TM)}[\sigma_{-6}(D_{F}^{-2-2})]\sigma(\xi)dx
=-\frac{\tilde{l}\times 2^{q}}{6\times (4\pi)^{3}}\int_{M}r_{M}\texttt{dvol}_{g}.
\end{equation}
Hence, we only need to compute $\int_{\partial M}\Phi$.

 Firstly, we give the symbol expansion of $D_{F}^{-2}$.  Recall the definition of the Dirac operator $D_{F}$.
 Let $\nabla^{TM}$ denote the Levi-civita connection about $g^M$. In the local coordinates $\{x_i; 1\leq i\leq n\}$ and
 the fixed orthonormal frame $\{\widetilde{e_1},\cdots,\widetilde{e_n}\}$, the connection matrix $(\omega_{s,t})$
is defined by
\begin{equation}
\nabla^{TM}(\widetilde{e_1},\cdots,\widetilde{e_n})= (\widetilde{e_1},\cdots,\widetilde{e_n})(\omega_{s,t}).
\end{equation}
Let $c(\widetilde{e_i})$ denotes the Clifford action.
Let $g^{ij}=g(dx_i,dx_j)$ and
\begin{equation}
\nabla^{TM}_{\partial_i}\partial_j=\sum_k\Gamma_{ij}^k\partial_k;
~\Gamma^k=g^{ij}\Gamma_{ij}^k.
\end{equation}
Let the cotangent vector $\xi=\sum \xi_jdx_j$ and $\xi^j=g^{ij}\xi_i$.
By the composition formula of psudodifferential operators in \cite{WW}
and direct computations, we obtain

\begin{lem}
 Let $D_{F}$ be the sub-Dirac operator associated to $g$ on  $\Gamma({S(F)\otimes\wedge(F^{\perp,\star})})$.
Then
\begin{eqnarray}
\sigma_{-2}(D_{F}^{-2})&=&|\xi|^{-2}; \\
\sigma_{-3}(D_{F}^{-2})&=&-\sqrt{-1}|\xi|^{-4}\xi_k\Big(\Gamma^k-2\sigma^k\otimes {\rm Id}_{\wedge(F^{\perp,\star})}
      -{\rm Id}_{S(F)}\otimes2\tilde{\sigma}^k\nonumber\\
      &&-\frac{1}{2}\sum_{i,j=1}^{2p}\sum_{s=1}^{q}\langle\nabla^{TM}_{\partial_k}f_{j}, h_{s}\rangle c(f_{j})c(h_{s})
-\frac{1}{2}\sum_{s,t=1}^{q}\sum_{i=1}^{2p}\langle\nabla^{TM}_{\partial_k}f_{j}, h_{s}\rangle c(f_{j})c(h_{s})\Big) \nonumber\\
 &&-\sqrt{-1}|\xi|^{-6}2\xi^j\xi_\alpha\xi_\beta
\partial_jg^{\alpha\beta}.
\end{eqnarray}
where $\sigma_k=-\frac{1}{4}\sum_{k,l}\omega_{k,l}(\partial_k)c(f_{k})c(f_{l})$,
$\tilde{\sigma}_k=\frac{1}{4}\sum_{r,t}\omega_{r,t}(\partial_k)[\bar{c}(h_{r})\bar{c}(h_{t})-c(h_{r})c(h_{t})] $.
\end{lem}
\begin{proof}
In the  fixed orthonormal frame $\{\widetilde{e_1},\cdots,\widetilde{e_n}\}$, the Bochner Laplacian $\triangle^{F}$ stating that in \cite{BGV},
\begin{equation}
\Delta^{S(F)\otimes\wedge(F^{\perp,\star})}=-\sum_{i,j}g^{ij}(x)\Big(\nabla_{\partial_{i}}^{S(F)\otimes\wedge(F^{\perp,\star})}
\nabla_{\partial_{j}}^{S(F)\otimes\wedge(F^{\perp,\star})}
-\sum_{k}\Gamma_{ij}^{k}\nabla_{\partial_{k}}^{S(F)\otimes\wedge(F^{\perp,\star})}\Big).
\end{equation}
Let $\sigma_i=-\frac{1}{4}\sum_{k,l}\omega_{k,l}(\partial_i)c(f_{k})c(f_{l})$,
$\tilde{\sigma}_i=\frac{1}{4}\sum_{r,t}\omega_{r,t}(\partial_i)[\bar{c}(h_{r})\bar{c}(h_{t})-c(h_{r})c(h_{t})] $ and $\sigma^j=g^{ij}\sigma_i$.

Combining (2.5), (2.10) and (4.6), we have
\begin{eqnarray}
D^2_{F}&=&-\sum_{i,j}g^{ij}(x)\Big\{\Big[\partial_{i}+\sigma_i\otimes {\rm Id}_{\wedge(F^{\perp,\star})}
      +{\rm Id}_{S(F)}\otimes\tilde{\sigma}_i+\frac{1}{2}\sum_{j=1}^{p}\sum_{s=1}^q<S(\partial_{i})f_j,h_s>c(f_j)c(h_s)\Big]\nonumber\\
&&\times\Big[\partial_{j}+\sigma_j\otimes {\rm Id}_{\wedge(F^{\perp,\star})}
      +{\rm Id}_{S(F)}\otimes\tilde{\sigma}_j+\frac{1}{2}\sum_{j=1}^{p}\sum_{s=1}^q<S(\partial_{j})f_j,h_s>c(f_j)c(h_s)\Big]\nonumber\\
&&-\sum_{k}\Gamma_{ij}^{k} \Big[\partial_{k}+\sigma_k\otimes {\rm Id}_{\wedge(F^{\perp,\star})}
      +{\rm Id}_{S(F)}\otimes\tilde{\sigma}_k+\frac{1}{2}\sum_{j=1}^{p}\sum_{s=1}^q<S(\partial_{k})f_j,h_s>c(f_j)c(h_s)\Big]
\Big\}\nonumber\\
&&+\frac{r_M}{4}+\frac{1}{4}\sum_{i=1}^{p}\sum_{r,s,t=1}^q\left<R^{F^\bot}(f_i,h_r)h_t,h_s\right>
c(f_i)c(h_r)\widehat{c}(h_s)\widehat{c}(h_t)\nonumber\\
&&+\frac{1}{8}\sum_{i,j=1}^{p}\sum_{s,t=1}^q\left<R^{F^\bot}(f_i,f_j)h_t,h_s\right>
c(f_i)c(f_j)\widehat{c}(h_s)\widehat{c}(h_t)\nonumber\\
&&+\frac{1}{8}\sum_{s,t,r,u=1}^q\left<R^{F^\bot}(h_r,h_l)h_t,h_s\right>
c(h_r)c(h_u)\widehat{c}(h_s)\widehat{c}(h_t).
\end{eqnarray}
From (4.5) in \cite{WW}, we have
\begin{equation}
\sigma_{2}(D_{F}^{2})\sigma_{-2}(D_{F}^{-2})=1;
 \  \sigma_{1}(D_{F}^{2})\sigma_{-2}(D_{F}^{-2})+\sigma_{2}(D_{F}^{2})\sigma_{-3}(D_{F}^{-2})+\sum_{j}\partial_{\xi_{j}}
 \sigma_{2}(D_{F}^{2})D_{x_{j}}\sigma_{-2}(D_{F}^{-2})=0.
\end{equation}
Then the Lemma follows.
\end{proof}

Since $\Phi$ is a global form on $\partial M$, so for any fixed point $x_0\in\partial M$, we can choose the normal coordinates
$U$ of $x_0$ in $\partial M$ (not in $M$) and compute $\Phi(x_0)$ in the coordinates $\widetilde{U}=U\times [0,1)\subset M$ and the
metric $\frac{1}{h(x_n)}g^{\partial M}+dx_n^2.$ For details, see Section 2.2.2 in \cite{Wa3}.

Let us now turn to compute $\Phi$ (see formula (3.4) for the definition of $\Phi$), since the sum is taken over $
-r-l+k+j+|\alpha|=-5,~~r,l\leq-2,$ then we have the following five cases:

{\bf case a)~I)}~$r=-2,~l=-2~k=j=0,~|\alpha|=1$

From (3.4) we have
 \begin{equation}
{\rm case~a)~I)}=-\int_{|\xi'|=1}\int^{+\infty}_{-\infty}\sum_{|\alpha|=1}
{\rm trace} [\partial^\alpha_{\xi'}\pi^+_{\xi_n}\sigma_{-2}(D_{F}^{-2})\times
\partial^\alpha_{x'}\partial_{\xi_n}\sigma_{-2}(D_{F}^{-2})](x_0)d\xi_n\sigma(\xi')dx',
\end{equation}
By Lemma 2.2 in \cite{Wa3}, for $i<n$, then
\begin{equation}
\partial_{x_i}\sigma_{-2}(D^{-2})(x_0)=\partial_{x_i}{(|\xi|^{-2})}(x_0)=
-\frac{\partial_{x_i}(|\xi|^2)(x_0)}{|\xi|^4}=0,
\end{equation}
so case a) I) vanishes.

 {\bf case a)~II)}~$r=-2,~l=-2~k=|\alpha|=0,~j=1$

From (3.4) we have
 \begin{equation}
{\rm case~a)~II)}=-\frac{1}{2}\int_{|\xi'|=1}\int^{+\infty}_{-\infty} {\rm
trace} [\partial_{x_n}\pi^+_{\xi_n}\sigma_{-2}(D_{F}^{-2})\times
\partial_{\xi_n}^2\sigma_{-2}(D_{F}^{-2})](x_0)d\xi_n\sigma(\xi')dx',
\end{equation}
An application of Lemma 2.2 in \cite{Wa3}, shows that
 \begin{equation}
\partial_{x_n}\sigma_{-2}(D_{F}^{-2})(x_0)|_{|\xi'|=1}=-\frac{h'(0)}{(1+\xi_n^2)^2}.
\end{equation}
By (2.1.1) in \cite{Wa3} and the Cauchy integral formula, we obtain
\begin{eqnarray}
&&\pi^+_{\xi_n}\partial_{x_n}\sigma_{-2}(D_{F}^{-2})(x_0)|_{|\xi'|=1}  \nonumber\\
&=&-h'(0)\frac{1}{2\pi i}\lim_{u\rightarrow 0^-}\int_{\Gamma^+}\frac{\frac{1}{(\eta_n+i)^2(\xi_n+iu-\eta_n)}}
{(\eta_n-i)^2}d\eta_n  \nonumber\\
&=&h'(0)\frac{i\xi_n+2}{4(\xi_n-i)^2},
\end{eqnarray}
and
\begin{equation}
\partial^2_{\xi_n}(|\xi|^{-2})(x_0)=\frac{-2+6\xi_n^2}{(1+\xi_n^2)^3}.
\end{equation}
We note that
\begin{eqnarray}
&&\int_{-\infty}^{\infty}\frac{i\xi_n+2}{(\xi_n-i)^2}\times
\frac{-1+3\xi_n^2}{(1+\xi_n^2)^3}d\xi_n  \nonumber\\
&=&\int_{\Gamma^+}\frac{3i\xi_n^3+6\xi_n^2-i\xi_n-2}{(\xi_n-i)^5(\xi_n+i)^3}d\xi_n  \nonumber\\
&=&\frac{2\pi
i}{4!}\left[\frac{3i\xi_n^3+6\xi_n^2-i\xi_n-2}{(\xi_n+i)^3}\right]^{(4)}|_{\xi_n=i} \nonumber\\
&=&\frac{5\pi}{16}
\end{eqnarray}
Since $n=2p+q=6$, ${\rm tr}_{({S(F)\otimes\wedge(F^{\perp,\star})})}[{\rm id}]=\tilde{l}\times 2^{q}$.
Combining (4.11) and (4.15), we have
{\bf case a} II)$=-\frac{5}{64}\tilde{l}\times 2^{q}\pi
h'(0)\Omega_4dx',$ where $\Omega_4$ is the
canonical volume of $S^4$.

 {\bf case a)~III)}~$r=-2,~l=-2~j=|\alpha|=0,~k=1$\\

 By (3.4) and an integration by parts, we obtain
 \begin{eqnarray}
{\rm case~ a)~III)}&=&-\frac{1}{2}\int_{|\xi'|=1}\int^{+\infty}_{-\infty}
{\rm trace} [\partial_{\xi_n}\pi^+_{\xi_n}\sigma_{-2}(D_{F}^{-2})\times
\partial_{\xi_n}\partial_{x_n}\sigma_{-2}(D_{F}^{-2})](x_0)d\xi_n\sigma(\xi')dx'  \nonumber\\
 &=&\frac{1}{2}\int_{|\xi'|=1}\int^{+\infty}_{-\infty} {\rm trace}
[\partial_{\xi_n}^2\pi^+_{\xi_n}\sigma_{-2}(D_{F}^{-2})\times
\partial_{x_n}\sigma_{-2}(D_{F}^{-2})](x_0)d\xi_n\sigma(\xi')dx'.
\end{eqnarray}
From Lemma 2.2 in \cite{Wa3}, we have
 \begin{equation}
\partial_{\xi_n}^2\pi_{\xi_n}^+\sigma_{-2}(D_{F}^{-2})(x_0)|_{|\xi'|=1}=\frac{-i}{(\xi_n-i)^3}.
\end{equation}
Combining (4.12) and (4.17), we have
\begin{equation}
{\rm {\bf case~a)~III)}}=4ih'(0)\int_{|\xi'|=1}\int^{+\infty}_{-\infty}
\int_{\Gamma^+}\frac{1}{(\xi_n-i)^5(\xi_n+i)^2}d\xi_n\sigma(\xi')dx'=\frac{5}{64}\tilde{l}\times 2^{q}\pi h'(0)\Omega_4dx'.
\end{equation}

Thus the sum of {\bf case~a)~II)} and {\bf case~ a)~III)} is zero.

 {\bf case b)}~$r=-2,~l=-3,~k=j=|\alpha|=0$

 By (3.4) and an integration by parts, we get
\begin{eqnarray}
{\rm case~ b)}&=&-i\int_{|\xi'|=1}\int^{+\infty}_{-\infty}
{\rm trace} [\pi^+_{\xi_n}\sigma_{-2}(D_{F}^{-2})\times
\partial_{\xi_n}\sigma_{-3}(D_{F}^{-2})](x_0)d\xi_n\sigma(\xi')dx'  \nonumber\\
&=& i\int_{|\xi'|=1}\int^{+\infty}_{-\infty}
{\rm trace} [\partial_{\xi_n}\pi^+_{\xi_n}\sigma_{-2}(D_{F}^{-2})\times
\sigma_{-3}(D_{F}^{-2})](x_0)d\xi_n\sigma(\xi')dx'.
\end{eqnarray}
By Lemma 2.2 in \cite{Wa3}, we have
\begin{equation}
\partial_{\xi_n}\pi_{\xi_n}^+\sigma_{-2}(D_{F}^{-2})(x_0)|_{|\xi'|=1}=\frac{i}{2(\xi_n-i)^2}.
\end{equation}

In the normal coordinate, $g^{ij}(x_0)=\delta_i^j$ and $\partial_{x_j}(g^{\alpha\beta})(x_0)=0,$ {\rm if
}$j<n;~=h'(0)\delta^\alpha_\beta,~{\rm if }~j=n.$ So by Lemma A.2 in \cite{Wa3}, we have $\Gamma^n(x_0)=\frac{5}{2}h'(0)$ and
$\Gamma^k(x_0)=0$ for $k<n$. Let
\begin{equation}
\sigma_{-3}(D_{F}^{-2}):=A_{1}+A_{2},
\end{equation}
where
\begin{eqnarray}
A_{1}&=&\-\sqrt{-1}|\xi|^{-6}2\xi^j\xi_\alpha\xi_\beta\partial_jg^{\alpha\beta}; \\
A_{2}&=&-\sqrt{-1}|\xi|^{-4}\xi_k\Big(\Gamma^k+\frac{1}{2}\sum_{k,l}\omega_{k,l}(\partial^k)c(f_{k})c(f_{l})\otimes
{\rm Id}_{\wedge(F^{\perp,\star})} \nonumber\\
  &&-{\rm Id}_{S(F)}\otimes\frac{1}{2}\sum_{r,t}\omega_{r,t}(\partial^k)[\bar{c}(h_{r})\bar{c}(h_{t})-c(h_{r})c(h_{t})]
      -\sum_{i,j=1}^{2p}\sum_{s=1}^{q}\langle\nabla^{TM}_{\partial^k}f_{j}, h_{s}\rangle c(f_{j})c(h_{s})\Big).
\end{eqnarray}
Then
\begin{equation}
{\rm tr} [\partial_{\xi_n}\pi^+_{\xi_n}\sigma_{-2}(D_{F}^{-2})\times A_{1}]=\frac{i}{2(\xi_n-i)^2}\times\frac{-2ih'(0)\xi_n}{(1+\xi_n^2)^3}.
\end{equation}
By the trace identity  $\texttt{tr}(AB)=\texttt{tr}(BA)$, $\texttt{tr}(A\otimes B)=\texttt{tr}(A)\cdot\texttt{tr}(B)$ and the relation
of the Clifford action, we have
\begin{eqnarray}
&&\texttt{tr}[c(f_k)c(f_l)]=-\delta_k^{l}2^{p},\nonumber\\
&&\texttt{tr}[\bar{c}(h_r)\bar{c}(h_t)-c(h_r)c(h_t)]=2\delta_r^{t}2^{q}.
\end{eqnarray}
Then
\begin{eqnarray}
{\rm tr} [\partial_{\xi_n}\pi^+_{\xi_n}\sigma_{-2}(D_{F}^{-2})\times A_{2}]
&=&\frac{\xi_n}{2(\xi_n-i)^4(\xi_n+i)^2} \Big(\frac{5}{2}h'(0)+\frac{1}{2}\sum_{k}\omega_{k,k}(\partial^n)2^{p+q}
     -\sum_{r}\omega_{r,r}(\partial^n)2^{p+q} \Big).\nonumber\\
&=&\frac{5h'(0)\xi_n}{4(\xi_n-i)^4(\xi_n+i)^2} ,
\end{eqnarray}
where we have used the fact that when $k=l$ and $r=t$, $\sum_{k}\omega_{k,l}(\partial^n)=\sum_{r}\omega_{r,t}(\partial^n)=0$.

Hence in this case,
\begin{eqnarray}
{\bf case~ b)}&=& i \int_{|\xi'|=1}\int^{+\infty}_{-\infty}{\rm trace}
 \Big[\frac{1}{2(\xi_n-i)^2}\times\Big(A_{1}+A_{2}\Big)\Big]d\xi_n\sigma(\xi')dx'  \nonumber\\
&=&\frac{ih'(0)}{4}\tilde{l}\times 2^{q}\Omega_4\int_{\Gamma^+}\frac{5\xi_n^3+9\xi_n}{(\xi_n-i)^5(\xi_n+i)^3}d\xi_ndx' \nonumber\\
&=&\frac{ih'(0)}{4}\tilde{l}\times 2^{q}\Omega_4 \frac{2\pi i}{4!}\Big[\frac{5\xi_n^3+9\xi_n}{(\xi_n+i)^3}\Big]^{(4)}|_{\xi_n=i}dx' \nonumber\\
&=&-\frac{15}{64}\tilde{l}\times 2^{q}\pi h'(0)\Omega_4dx'
\end{eqnarray}

{\bf  case c)}~$r=-3,~l=-2,~k=j=|\alpha|=0$

From (3.4) we have
 \begin{equation}
{\rm case~ c)}=-i\int_{|\xi'|=1}\int^{+\infty}_{-\infty}
{\rm trace} [\pi^+_{\xi_n}\sigma_{-3}(D_{F} ^{-2})\times
\partial_{\xi_n}\sigma_{-2}(D_{F} ^{-2})](x_0)d\xi_n\sigma(\xi')dx'.
\end{equation}

By the Leibniz rule, trace property and "++" and "-~-" vanishing after the integration over $\xi_n$ in \cite{Wa4}, then
\begin{eqnarray}
&&\int^{+\infty}_{-\infty}{\rm trace}
[\pi^+_{\xi_n}\sigma_{-3}(D_{F} ^{-2})\times
\partial_{\xi_n}\sigma_{-2}(D_{F} ^{-2})]d\xi_n  \nonumber\\
&=&\int^{+\infty}_{-\infty}{\rm tr} [\sigma_{-3}(D_{F} ^{-2})\times\partial_{\xi_n}\sigma_{-2}(D_{F} ^{-2})]d\xi_n
-\int^{+\infty}_{-\infty}{\rm tr}[\pi^-_{\xi_n}\sigma_{-3}(D_{F} ^{-2})\times\partial_{\xi_n}\sigma_{-2}(D_{F} ^{-2})]d\xi_n  \nonumber\\
&=&\int^{+\infty}_{-\infty}{\rm tr} [\sigma_{-3}(D_{F} ^{-2})\times\partial_{\xi_n}\sigma_{-2}(D_{F} ^{-2})]d\xi_n-\int^{+\infty}_{-\infty}{\rm
tr} [\pi^-_{\xi_n}\sigma_{-3}(D_{F}^{-2})\times\partial_{\xi_n}\pi^+_{\xi_n}\sigma_{-2}(D_{F} ^{-2})]d\xi_n  \nonumber\\
&=&\int^{+\infty}_{-\infty}{\rm tr} [\sigma_{-3}(D_{F} ^{-2})\times\partial_{\xi_n}\sigma_{-2}(D_{F} ^{-2})]d\xi_n-\int^{+\infty}_{-\infty}{\rm
tr} [\sigma_{-3}(D_{F} ^{-2})\times\partial_{\xi_n}\pi^+_{\xi_n}\sigma_{-2}(D_{F} ^{-2})]d\xi_n  \nonumber\\
&=&\int^{+\infty}_{-\infty}{\rm tr} [\partial_{\xi_n}\sigma_{-2}(D_{F} ^{-2})\times\sigma_{-3}(D_{F} ^{-2})]d\xi_n+\int^{+\infty}_{-\infty}{\rm
tr}[\partial_{\xi_n}\sigma_{-3}(D_{F}^{-2})\times\pi^+_{\xi_n}\sigma_{-2}(D_{F}^{-2})]d\xi_n  \nonumber\\
&=&\int^{+\infty}_{-\infty}{\rm tr} [\partial_{\xi_n}\sigma_{-2}(D_{F} ^{-2})\times\sigma_{-3}(D_{F} ^{-2})]d\xi_n+\int^{+\infty}_{-\infty}{\rm
tr}[\pi^+_{\xi_n}\sigma_{-2}(D_{F} ^{-2})\times\partial_{\xi_n}\sigma_{-3}(D_{F} ^{-2})]d\xi_n  .
\end{eqnarray}
Then we have
\begin{equation}
{\rm {\bf case~ c)}}={\rm {\bf case~ b)}}-i\int_{|\xi'|=1}\int^{+\infty}_{-\infty}{\rm
tr}[\partial_{\xi_n}\sigma_{-2}(D_{F} ^{-2})\times\sigma_{-3}(D_{F} ^{-2})]d\xi_n\sigma(\xi')dx'.
\end{equation}
 In order to compute {\bf case c)}, we only need compute the last term in (4.30).

From (4.4), one sees that
\begin{equation}
\partial_{\xi_n}\sigma_{-2}(D_{F} ^{-2})(x_0)|_{|\xi'|=1}=-\frac{2\xi_n}{(\xi_n^2+1)^2}.
\end{equation}
Similarly to \textbf{case b}, we obtain
\begin{eqnarray}
&&-i\int_{|\xi'|=1}\int^{+\infty}_{-\infty}{\rm tr}[\partial_{\xi_n}\sigma_{-2}(D_{F} ^{-2})\times \sigma_{-3}(D_{F} ^{-2})]
d\xi_n\sigma(\xi')dx' \nonumber\\
&=&-i\int_{|\xi'|=1}\int^{+\infty}_{-\infty}{\rm tr}[\partial_{\xi_n}\sigma_{-2}(D_{F} ^{-2})\times (A_{1}+A_{2})]d\xi_n\sigma(\xi')dx'
 \nonumber\\
&=&h'(0)\tilde{l}\times 2^{q}\Omega_4\int_{\Gamma^+}\frac{5\xi_n^4+9\xi_n^{2}}{(\xi_n-i)^5(\xi_n+i)^5}d\xi_ndx'   \nonumber\\
&=&h'(0)\tilde{l}\times 2^{q}\Omega_4 \frac{2\pi i}{4!}\Big[\frac{5\xi_n^4+9\xi_n^{2}}{(\xi_n+i)^5}\Big]^{(4)}|_{\xi_n=i}dx'  \nonumber\\
&=&\frac{15}{32}\tilde{l}\times 2^{q} \pi h'(0)\Omega_4dx'.
\end{eqnarray}
Combining (4.27), (4.30) and (4.32), we have
 the sum of {\bf case b)} and {\bf case c)} is
 zero. Now $\Phi$ is the sum of the cases a), b) and c), so is zero. Hence we conclude that

\begin{thm}
Let M be a 6-dimensional compact connected foliation with the boundary $\partial M$ and the metric $g^{M}$ as above ,
and $D_{F}$ be the sub-Dirac operator on  $\Gamma({S(F)\otimes\wedge(F^{\perp,\star})})$, then
\begin{equation}
Vol_{6}^{(2,2)}(M,F)=-\frac{\tilde{l}\times2^{q}}{6\times (4\pi)^{3}}\int_{M}r_{M}\texttt{dvol}_{g}.
\end{equation}
where $r_M$ be the scaler curvature of the foliation and $dim(S(F))=\tilde{l}$.
\end{thm}

\begin{rem}
Let M be a 6-dimensional compact connected foliation with the boundary $\partial M$ and the metric $g^{M}$ as above,  and $D_{F}$ be
 the sub-Dirac operator. When $n=p$, we obtain Theorem 1 in \cite{Wa4}.
\end{rem}

\subsection{The gravitational action for $6$-dimensional manifolds with boundary}
Let $M$ be $6$-dimensional manifolds with boundary and $P,P'$ be two pseudodifferential operators with transmission
property on $\widehat M$. Motivated by (4) in \cite{Wa4}, we define locally
 \begin{equation}
{\rm res}_{2,2}(P,P'):=-\frac{1}{2}\int_{|\xi'|=1}\int^{+\infty}_{-\infty}
{\rm trace} [\partial_{x_n}\pi^+_{\xi_n}\sigma_{-2}(P)\times
\partial_{\xi_n}^2\sigma_{-2}(P')]d\xi_n\sigma(\xi')dx';
\end{equation}
 \begin{equation}
{\rm res}_{2,3}(P,P'):=-i\int_{|\xi'|=1}\int^{+\infty}_{-\infty}
{\rm trace} [\pi^+_{\xi_n}\sigma_{-2}(P)\times
\partial_{\xi_n}\sigma_{-3}(P')]d\xi_n\sigma(\xi')dx'.
\end{equation}
Combining  (4.34) and (4.35), we have
\begin{equation}
 {\rm case~ a)~ II)}={\rm res}_{2,2}(D_{F}^{-2},D_{F}^{-2});~{\rm case~ b)}={\rm res}_{2,3}(D_{F}^{-2},D_{F}^{-2}).
\end{equation}
 Now, we assume $\partial M$ is flat, then
$\{dx_i=e_i\},~g^{\partial M}_{i,j}=\delta_{i,j},~\partial
_{x_s}g^{\partial M}_{i,j}=0$. So ${\rm res}_{2,2}(D_{F}^{-2},D_{F}^{-2})$
and ${\rm res}_{2,3}(D_{F}^{-2},D_{F}^{-2})$ are two global forms locally
defined by the aboved oriented orthonormal basis $\{dx_i\}$. From case
a) II) and case b), we have

\begin{thm}
Let M be a 6-dimensional compact connected foliation with the boundary $\partial M$ and the metric $g^{M}$ as above ,
and $D_{F}$ be the sub-Dirac operator on  $\Gamma({S(F)\otimes\wedge(F^{\perp,\star})})$. Assume $\partial M$ is flat, then
\begin{eqnarray}
&&\int_{\partial M}{\rm res}_{2,2}(D_{F}^{-2},D_{F}^{-2})=\frac{1}{64}\tilde{l}\times2^{q}\Omega_4I_{\rm {Gr,b}}; \\
&&\int_{\partial M}{\rm res}_{2,3}(D_{F}^{-2},D_{F}^{-2})=\frac{3}{64}\tilde{l}\times2^{q}\Omega_4I_{\rm {Gr,b}} .
\end{eqnarray}
\end{thm}

\section{A Kastler-Kalau-Walze Type Theorem for 5-dimensional foliations with boundary}

\subsection{A Kastler-Kalau-Walze Type Theorem for 5-dimensional foliations with boundary}

First of all,  for $5$-dimensional foliations with boundary, we compute
${\rm Vol}^{(2,2)}_5$. From Theorem 2.9, we have
\begin{equation}
\widetilde{{\rm Wres}}[(\pi^+D_{F}^{-2})^2]=\int_{\partial M}\Phi.
\end{equation}

 From (3.4), when $n=5$, $ r-k-|\alpha|+l-j-1=-5,~~r,l\leq-2$, so
we get $r=l=-2,~k=|\alpha|=j=0,$ then
\begin{equation}
\Phi=\int_{|\xi'|=1}\int^{+\infty}_{-\infty}
[ \sigma^+_{-2}(D_{F}^{-2})(x',0,\xi',\xi_n)\times\partial_{\xi_n}\sigma_{-2}(D_{F}^{-2})(x',0,\xi',\xi_n)]d\xi_5\sigma(\xi')dx'.
\end{equation}
An application of Lemma 2.2 in \cite{Wa3}, shows that
\begin{equation}
\pi^+_{\xi_n}\sigma_{-2}(x_0)|_{|\xi'|=1}=\frac{1}{2i(\xi_n-i)}.
\end{equation}
 By (3.21) and ${\rm tr}_{({S(F)\otimes\wedge(F^{\perp,\star})})}[{\rm id}]=\tilde{l}\times2^{q}$, we obtain
 \begin{equation}
{\rm Vol}^{(2,2)}_5=\frac{\pi i}{8}\tilde{l}\times2^{q}\Omega_3{\rm Vol}_{\partial M}.
\end{equation}
 By $I_{\rm {Gr,b}}=-4h'(0){\rm Vol}_{\partial M}$, we have

 \begin{thm}
Let M be a 5-dimensional compact connected foliation with the boundary $\partial M$ and the metric $g^{M}$ as above ,
and $D_{F}$ be the sub-Dirac operator on  $\Gamma({S(F)\otimes\wedge(F^{\perp,\star})})$, then
\begin{eqnarray}
&&{\rm Vol}^{(2,2)}_5=\widetilde{{\rm Wres}}[(\pi^+D_{F}^{-2})^2]=\frac{\pi i}{8}\tilde{l}\times2^{q}\Omega_3{\rm Vol}_{\partial M}; \\
&& I_{\rm {Gr,b}}=\frac{32ih'(0)}{\tilde{l}\times2^{q}\pi\Omega_3}\widetilde{{\rm Wres}}[(\pi^+D_{F}^{-2})^2],
\end{eqnarray}
where ${\rm Vol}_{\partial M}$ denotes the canonical volume of ${\partial M}.$
\end{thm}

\subsection{A Kastler-Kalau-Walze Type Theorem for 5-dimensional Manifolds with boundary}

In this section, we compute the lower dimension volume for 5-dimension spin manifolds with boundary and get a Kastler-Kalau-Walze type
Formula in this case.

Let $M$ be an n-dimensional compact oriented connected manifold with boundary $\partial M$. We assume that the metric
$g^{M}$ on $M$ has the following form near the boundary
 \begin{equation}
 g^{M}=\frac{1}{h(x_{n})}g^{\partial M}+\texttt{d}x _{n}^{2},
\end{equation}
where $g^{\partial M}$ is the metric on $\partial M$; $h(x_{n})\in C^{\infty}([0,1))=\{g|_{[0,1)}| g\in C^{\infty}((-\varepsilon,1))\}$
for some sufficiently small $\varepsilon>0$ and satisfied $h(x_{n})>0, \ h(0)=1$, where $x_{n}$ denotes the normal directional coordinate.
Let $D$ be the Dirac operator associated to $g^{M}$
on $C^{\infty}(S(TM))$. Let $p_{1},p_{2}$ be nonnegative integers and $p_{1}+p_{2}\leq n$. Then Wang pointed out
the related definition between the volumes of compact connected manifolds and the Wodzicki residue in \cite{Wa1}.
\begin{defn} Lower-dimensional volumes of compact connected manifolds with boundary are defined by
 \begin{equation}\label{}
   Vol_{n}^{(p_{1},p_{2})}M:=\widetilde{Wres}[\pi^{+}D^{-p_{1}} \circ\pi^{+}D^{-p_{2}}] .
\end{equation}
\end{defn}

  Denote by $\sigma_{l}(D)$ the \emph{l}-order symbol of an operator $D$.
Combining (2.1.4)-(2.1.8) of \cite{Wa3}, we obtain
 \begin{equation}\label{}
 \widetilde{{\rm Wres}}[\pi^+D^{-p_1}\circ\pi^+D^{-p_2}]=\int_M\int_{|\xi|=1}{\rm
trace}_{S(TM)}[\sigma_{-n}(D^{-p_1-p_2})]\sigma(\xi)dx+\int_{\partial M}\Phi,
\end{equation}
where
 \begin{eqnarray}
\Phi&=&\int_{|\xi'|=1}\int_{-\infty}^{+\infty}\sum_{j,k=0}^{\infty}\sum \frac{(-i)^{|\alpha|+j+k+1}}{\alpha!(j+k+1)!}
    \text{trace}_{S(TM)}[\partial_{x_{n}}^{j}\partial_{\xi'}^{\alpha}\partial_{\xi_{n}}^{k}\sigma_{r}^{+}
    (D^{-p_{1}})(x',0,\xi',\xi_{n})\nonumber\\
&&\times\partial_{x'}^{\alpha}\partial_{\xi_{n}}^{j+1}\partial_{x_{n}}^{k}\sigma_{l}(D^{-p_{2}})(x',0,\xi',\xi_{n})
\texttt{d}\xi_{n}\sigma(\xi')\texttt{d}x' ,
\end{eqnarray}
and the sum is taken over $r-k+|\alpha|+\ell-j-1=-n, \ r\leq-p_{1}, \ \ell\leq-p_{2}$.

Since $[\sigma_{-n}(\Delta_{k}^{-p_{1}-p_{2}})]|_{M}$ has the same expression with the case of  without boundary in \cite{Wa3},
so locally we can use Theorem 2.9 to compute the first term. Let us now give explicit formulas for the $Vol_{g}^{(1)}M$ in dimension 5.

Since $\Gamma(\frac{1}{2}+1)=\frac{1}{2}\Gamma(\frac{1}{2})=\frac{\sqrt{\pi}}{2}$ and $\Gamma(\frac{5}{2}+1)=\frac{5}{2}\Gamma(\frac{3}{2}+1)
=\frac{15}{4}\cdot\Gamma(\frac{1}{2}+1)=\frac{15\sqrt{\pi}}{4}$, we get
\begin{equation}
v_{5,1}=\frac{1}{5}2^{\frac{(1-5)(5+1)}{10}}\pi^{\frac{1-5}{2}}\frac{\Gamma(\frac{5}{2}+1)^{\frac{1}{2}}}{\Gamma(\frac{k}{2}+1)}
=\frac{1}{5}2^{-\frac{12}{5}}\pi^{-2}\frac{(15\sqrt{\pi}/4)^{\frac{1}{5}}}{\sqrt{\pi}/2}=\frac{\pi\sqrt[5]{30}}{20\sqrt[10]{\pi}}.
\end{equation}
Therefore, by using Proposition 2.3 in \cite{RP}, we see that in dimension 5
\begin{equation}
\int_M\int_{|\xi|=1}{\rm trace}_{S(TM)}[\sigma_{-5}(D^{-2-1})]\sigma(\xi)dx
 =-\frac{\pi\sqrt[5]{30}}{240\sqrt[10]{\pi}2^p\pi^{p+\frac{q}{2}}}\int_Mr_M\texttt{d}vol,
\end{equation}
where we have used the fact that $\int_{M}\Delta_{g}k\texttt{d}v_{g}(x)=\int_{M}g(\nabla k,\nabla 1)\texttt{d}v_{g}(x)=0$.
Hence we only need to compute $\int_{\partial M}\Phi$.

Since $\Phi$ is a global form on $\partial M$, so for any fixed point $x_{0}\in\partial M$, we can choose the normal coordinates $u$
of $x_{0}$ in $\partial M$ and compute $\Phi(x_{0})$ in the coordinates $\tilde{U}=U\times[0,1)\subset M$ and the metric
$g^{M}=\frac{1}{h(x_{n})}g^{\partial M}+\texttt{d}x _{n}^{2}$.

Firstly, we recall the symbol expansion of $D^{-2}$ and $D^{-1}$. By Lemma 1 in \cite{Wa4} and Lemma 2.1 in \cite{Wa3}, we have
\begin{lem}\cite{Wa4}\cite{Wa3}
 Let $D$ be the Dirac operator associated to $g$ on the spinors bundle $S(TM)$. Then
\begin{eqnarray}
&&\sigma_{-1}(D^{-1})=\frac{\sqrt{-1}c(\xi)}{|\xi|^{2}}; \\
&& \sigma_{-2}(D^{-1})=\frac{c(\xi)p_0c(\xi)}{|\xi|^4}+\frac{c(\xi)}{|\xi|^6}\sum_jc(dx_j)
[\partial_{x_j}[c(\xi)]|\xi|^2-c(\xi)\partial_{x_j}(|\xi|^2)];\\
&&\sigma_{-2}(D^{-2})=|\xi|^{-2};\\
&&\sigma_{-3}(D^{-2})=-\sqrt{-1}|\xi|^{-4}\xi_k(\Gamma^k-2\delta^k)-\sqrt{-1}|\xi|^{-6}2\xi^j\xi_\alpha\xi_\beta
\partial_jg^{\alpha\beta}.
\end{eqnarray}
where $p_0=-\frac{1}{4}\sum_{s,t}\omega_{s,t}(\partial_i)c(\widetilde{e_s})c(\widetilde{e_t}).$
\end{lem}

Let us now turn to compute $\Phi$ (see formula (5.10) for definition of $\Phi$). Since the sum is taken over $-r-\ell+k+j+|\alpha|=4,
 \ r\leq-2, \ell\leq-1$, then we have the following five cases:

\textbf{Case a (I)}: \ $r=-2, \ \ell=-1, \ k=j=0, \ |\alpha|=1$

From (5.10) we have
\begin{equation}
\text{ Case a (\text{I}) }=-\int_{|\xi'|=1}\int_{-\infty}^{+\infty}\sum_{|\alpha|=1}\text{trace}
   [\partial_{\xi'}^{\alpha}\pi_{\xi_{n}}^{+}\sigma_{-2}(D^{-2})
\times \partial_{x'}^{\alpha}\partial_{\xi_{n}}\sigma_{-1}(D^{-1})](x_{0})\texttt{d}\xi_{n}
\sigma(\xi')\texttt{d}x'.
\end{equation}
Then an application of Lemma 2.2 in \cite{Wa3} shows, for $i<n$
\begin{equation}
\partial_{x_i}\sigma_{-1}(D^{-1})(x_0)=\partial_{x_i}\left(\frac{\sqrt{-1}c(\xi)}{|\xi|^2}\right)(x_0)=
\frac{\sqrt{-1}\partial_{x_i}[c(\xi)](x_0)}{|\xi|^2}-\frac{\sqrt{-1}c(\xi)\partial_{x_i}(|\xi|^2)(x_0)}{|\xi|^4}=0,
\end{equation}
so \textbf{Case a (I)} vanishes.

\textbf{Case a (II)}: \ $r=-2, \ \ell=-1, \ k=|\alpha|=0, \ j=1$

From (5.10) we have
\begin{equation}
\text{ Case a (\text{II}) }=-\frac{1}{2}\int_{|\xi'|=1}\int_{-\infty}^{+\infty}
             \text{trace}[\partial_{x_{n}}\pi_{\xi_{n}}^{+}\sigma_{-2}(D^{-2})
\times \partial_{\xi_{n}}^{2}\sigma_{-1}(D^{-1})](x_{0})\texttt{d}\xi_{n}\sigma(\xi')\texttt{d}x'.
\end{equation}
By (2.2.16) in \cite{Wa3}, we have
\begin{equation}
\partial_{x_{n}}\sigma_{-2}(D^{-2})(x_{0})|_{|\xi'|=1}=-\frac{h'(0)}{(1+\xi_{n}^{2})^{2}}.
\end{equation}
 By the Cauchy integral formula we obtain
\begin{equation}
\pi_{\xi_{n}}^{+}[\frac{1}{(1+\xi_{n}^{2})^{2}}](x_{0})|_{|\xi'|=1}
  =\frac{1}{2\pi i}\lim_{u\rightarrow 0^{-}}\int_{\Gamma^{+}}\frac{\frac{1}{(\eta_{n}+i)^{2}(\xi_{n}+iu-\eta_{n})}}
{(\eta_{n}-i)^{2}}\texttt{d}\eta_{n}=-\frac{i\xi_{n}+2}{4(\xi_{n}-i)^{2}}.
\end{equation}
Then
\begin{equation}
\pi_{\xi_{n}}^{+}\partial_{x_{n}}\sigma_{-2}(D^{-2})(x_{0})|_{|\xi'|=1}= h'(0)\frac{i\xi_{n}+2}{4(\xi_{n}-i)^{2}},
\end{equation}
and
\begin{equation}
  \partial_{\xi_{n}}^{2}\sigma_{-1}(D^{-1})(x_{0})|_{|\xi'|=1}=\sqrt{-1}\left(-\frac{6\xi_nc(dx_n)+2c(\xi')}
{|\xi|^4}+\frac{8\xi_n^2c(\xi)}{|\xi|^6}\right).
\end{equation}

Since $n=5$, $\texttt{tr}(\texttt{id})=\texttt{dim}(S(TM))=4$. Locally $S(TM)|_{\widetilde {U}}\cong \widetilde {U}\times
\wedge^* _{\bf C}(\frac{n+1}{2})$, and ${\rm cl}_{\bf C}(n) \hookrightarrow{\rm cl}_{\bf C}(n+1)\cong {\rm Hom}(\wedge^*_{\bf C}(\frac{n+1}{2}))$. Let
$\{f_1,\cdots,f_6\}$ be the orthonormal basis of $\wedge^* _{\bf C}(\frac{n+1}{2}).$ Take a spin frame field $\sigma:~\widetilde
{U}\rightarrow {\rm Spin}(M)$ such that $\pi\sigma=\{\widetilde{e_1}\widetilde{e_6},\cdots,\widetilde{e_5}\widetilde{e_6}\}$ where $\pi :~{\rm
Spin}(M)\rightarrow O(M)$ is a double covering, then $\{[(\sigma,f_i)],~1\leq i\leq 6\}$ is an orthonormal frame of
$S(TM)|_{\widetilde {U}}.$  Let $\{E_1,\cdots,E_n\}$ be the canonical basis of ${\bf
R}^n$ and $c(E_i)\in {\rm cl}_{\bf C}(n)\cong {\rm Hom}(\wedge^*_{\bf C}(\frac{n+1}{2}),\wedge^* _{\bf C}(\frac{n+1}{2}))$ be the
Clifford action. In the following, since the global form $\Phi$ is independent of the choice of the local frame, so we can
compute ${\rm tr}_{S(TM)}$ in the frame $\{[(\sigma,f_i)],~1\leq i\leq 6\}.$

By the relation of the Clifford action and ${\rm tr}{AB}={\rm tr }{BA}$, we obtain the equalities:
\begin{equation}
 {\rm tr}[c(\xi')]=0;~~{\rm tr}[c(dx_n)]=0;~~{\rm tr}[c(\xi)](x_0)|_{|\xi'|=1}=0.
\end{equation}
Combining (5.21), (5.22) and (5.23), we obtain
\begin{eqnarray}
&& \text{trace}[\partial_{x_{n}}\pi_{\xi_{n}}^{+}\sigma_{-2}(D^{-2})\partial_{\xi_{n}}^{2}\sigma_{-1}(D^{-1})](x_{0})|_{|\xi'|=1}
\nonumber\\
&=&\sqrt{-1}h'(0)\text{trace}\left[\frac{i\xi_{n}+2}{4(\xi_{n}-i)^{2}}\times\left(-\frac{6\xi_nc(dx_n)+2c(\xi')}
{|\xi|^4}+\frac{8\xi_n^2c(\xi)}{|\xi|^6}\right)\right]\nonumber\\
&=&\sqrt{-1}h'(0)\frac{i\xi_{n}+2}{4(\xi_{n}-i)^{2}}\times\left(-\frac{6\xi_n\text{tr}[c(dx_n)]+2\text{tr}[c(\xi')]}
{|\xi|^4}+\frac{8\xi_n^2\text{tr}[c(\xi)]}{|\xi|^6}\right)\nonumber\\
&=&0.
\end{eqnarray}
Therefore $\textbf{Case a (II)}=0.$

\textbf{Case a (III)}: \ $r=-2, \ \ell=-1, \ j=|\alpha|=0, \ k=1$

 From (5.10) we have
  \begin{equation}
\text{ Case a (\text{III}) }=-\frac{1}{2}\int_{|\xi'|=1}\int_{-\infty}^{+\infty}\text{trace}[\partial_{\xi_{n}}\pi_{\xi_{n}}^{+}
\sigma_{-2}(D^{-2})\times\partial_{\xi_{n}}\partial_{x_{n}}\sigma_{-1}(D^{-1})](x_{0})
    \texttt{d}\xi_{n}\sigma(\xi')\texttt{d}x'.
 \end{equation}
An application of Lemma 2.2 in \cite{Wa3} shows
 \begin{equation}
\partial_{\xi_n}\partial_{x_n}\sigma_{-1}(D^{-1})(x_0)|_{|\xi'|=1}=-\sqrt{-1}h'(0)
\left[\frac{c(dx_n)}{|\xi|^4}-4\xi_n\frac{c(\xi')+\xi_nc(dx_n)}{|\xi|^6}\right]-
\frac{2\xi_n\sqrt{-1}\partial_{x_n}c(\xi')(x_0)}{|\xi|^4},
 \end{equation}
and
 \begin{equation}
\partial_{\xi_n}\pi_{\xi_n}^+\sigma_{-2}(D^{-2})(x_0)|_{|\xi'|=1}=\frac{i}{2(\xi_n-i)^2}.
 \end{equation}

From (5.24), (5.27) and (5.28), we have
 \begin{eqnarray}
 && \text{trace}[\partial_{\xi_{n}}\pi_{\xi_{n}}^{+}\sigma_{-2}(D^{-2})\times\partial_{\xi_{n}}\partial_{x_{n}}
     \sigma_{-1}(D^{-1})](x_{0})|_{|\xi'|=1}\nonumber\\
  &=& -\sqrt{-1}h'(0)\frac{i}{2(\xi_n-i)^2}\left[\frac{\text{tr}[c(dx_n)]}{|\xi|^4}-4\xi_n\frac{\text{tr}[c(\xi')]+\xi_n\text{tr}
  [c(dx_n)]}{|\xi|^6}\right]-
\frac{2\xi_n\sqrt{-1}\text{tr}[\partial_{x_n}c(\xi')(x_0)]}{|\xi|^4} \nonumber\\
    &=&0.
 \end{eqnarray}
Hence $\textbf{Case a (III)}=0.$

\textbf{Case b}: \ $r=-2, \ \ell=-2, \ k=j=|\alpha|=0$

From (5.10) we have
\begin{equation}
\text{ Case b}=-i\int_{|\xi'|=1}\int_{-\infty}^{+\infty}\text{trace}[\pi_{\xi_{n}}^{+}\sigma_{-2}(D^{-2})
       \times\partial_{\xi_{n}}\sigma_{-2}(D^{-1})](x_{0})\texttt{d}\xi_{n}\sigma(\xi')\texttt{d}x' .
\end{equation}
By (2.2.16) in \cite{Wa3}, we obtain
\begin{equation}
\pi_{\xi_{n}}^{+}\sigma_{-2}(D^{-2})(x_{0})|_{|\xi'|=1}=\frac{1}{2i(\xi_n-i)},
\end{equation}
and
\begin{eqnarray}
\partial_{\xi_n}\sigma_{-2}(D^{-1})(x_0)|_{|\xi'|=1}&=&
\frac{1}{(1+\xi_n^2)^3}[(2\xi_n-2\xi_n^3)c(dx_n)p_0c(dx_n)+(1-3\xi_n^2)c(dx_n)p_0c(\xi')\nonumber\\
&&+ (1-3\xi_n^2)c(\xi')p_0c(dx_n)-4\xi_nc(\xi')p_0c(\xi')\nonumber\\
&&+(3\xi_n^2-1)\partial_{x_n}c(\xi')-4\xi_nc(\xi')c(dx_n)\partial_{x_n}c(\xi')\nonumber\\
&&+2h'(0)c(\xi')+2h'(0)\xi_nc(dx_n)]
+6\xi_nh'(0)\frac{c(\xi)c(dx_n)c(\xi)}{(1+\xi^2_n)^4}.
\end{eqnarray}
Notice $p_0=-\frac{1}{4}\sum_{s,t}\omega_{s,t}(\partial_i)c(\widetilde{e_s})c(\widetilde{e_t}).$
By the relation of the Clifford action and ${\rm tr}{AB}={\rm tr }{BA}$, we obtain
\begin{eqnarray}
\text{tr}[c(\xi')c(dx_n)\partial_{x_n}c(\xi')]&=&\text{tr}[(c(\xi')e_{6})(c(dx_n)e_{6})\partial_{x_n}(c(\xi')e_{6})]\nonumber\\
&=&\text{tr}[(c(\xi')e_{6})c(dx_n)\partial_{x_n}(c(\xi'))]=\text{tr}[\partial_{x_n}(c(\xi'))(c(\xi')e_{6})c(dx_n)]\nonumber\\
&=&-\text{tr}[\partial_{x_n}(c(\xi')e_{6})c(dx_n)c(\xi')]=-\text{tr}[c(\xi')c(dx_n)\partial_{x_n}c(\xi')],
\end{eqnarray}
then $\text{tr}[c(\xi')c(dx_n)\partial_{x_n}c(\xi')]=0$.

Combining (5.31) and (5.32), we obtain
\begin{eqnarray}
&&\text{trace}[\pi_{\xi_{n}}^{+}\sigma_{-2}(D^{-2})
       \times\partial_{\xi_{n}}\sigma_{-2}(D^{-1})](x_{0})|_{|\xi'|=1}\nonumber\\
&=&\frac{1}{2i(\xi_n-i)(1+\xi_n^2)^3}\text{trace}\Big[
(2\xi_n-2\xi_n^3)c(dx_n)p_0c(dx_n)+(1-3\xi_n^2)c(dx_n)p_0c(\xi')\nonumber\\
&&+ (1-3\xi_n^2)c(\xi')p_0c(dx_n)-4\xi_nc(\xi')p_0c(\xi')+(3\xi_n^2-1)\partial_{x_n}c(\xi')-4\xi_nc(\xi')c(dx_n)\partial_{x_n}c(\xi')\nonumber\\
&&+2h'(0)c(\xi')+2h'(0)\xi_nc(dx_n)\Big]
+\frac{6\xi_nh'(0)}{2i(\xi_n-i)(1+\xi^2_n)^4}\text{trace}[c(\xi)c(dx_n)c(\xi)]\nonumber\\
&=&\frac{1}{2i(\xi_n-i)(1+\xi_n^2)^3}\Big[
(2\xi_n-2\xi_n^3)h'(0)\text{tr}[c(dx_n)]+(1-3\xi_n^2)h'(0)\text{tr}[c(\xi')]\nonumber\\
&&+ (1-3\xi_n^2)h'(0)\text{tr}[c(\xi')]+4\xi_nh'(0)\text{tr}[c(dx_n)]+(3\xi_n^2-1)\text{tr}[\partial_{x_n}c(\xi')]
-4\xi_n\text{tr}[c(\xi')c(dx_n)\partial_{x_n}c(\xi')]\nonumber\\
&&+2h'(0)c(\xi')+2h'(0)\xi_nc(dx_n)\Big]
+\frac{6\xi_nh'(0)}{2i(\xi_n-i)(1+\xi^2_n)^4}\text{tr}[c(\xi)c(dx_n)c(\xi)]\nonumber\\
&=&0
\end{eqnarray}
Then $\textbf{Case b}=0$.

\textbf{Case c}: \ $r=-3, \ \ell=-1, \ k=j=|\alpha|=0$

From(5.10) we have
\begin{equation}
\text{Case c}=-i\int_{|\xi'|=1}\int_{-\infty}^{+\infty}\text{trace}[\pi_{\xi_{n}}^{+}\sigma_{-3}(D^{-2})
 \times \partial_{\xi_{n}}\sigma_{-1}(D^{-1})](x_{0})\texttt{d}\xi_{n}\sigma(\xi')\texttt{d}x'  .
 \end{equation}

 Now, in the normal coordinate, $g^{ij}(x_0)=\delta_i^j$ and $\partial_{x_j}(g^{\alpha\beta})(x_0)=0,$ {\rm if
}$j<n;~=h'(0)\delta^\alpha_\beta,~{\rm if }~j=n.$ So by Lemma A.2 in \cite{Wa1}, we have $\Gamma^n(x_0)=\frac{5}{2}h'(0)$ and
$\Gamma^k(x_0)=0$ for $k<n$. By (20) in \cite{Wa4}, we have
\begin{eqnarray}
\sigma_{-3}(D^{-2})(x_0)|_{|\xi'|=1}
&=&-\sqrt{-1}|\xi|^{-4}\xi_k(\Gamma^k-2\delta^k)(x_0)|_{|\xi'|=1}-\sqrt{-1}|\xi|^{-6}2\xi^j\xi_\alpha\xi_\beta
\partial_jg^{\alpha\beta}(x_0)|_{|\xi'|=1}\nonumber\\
&=&\frac{-i}{(1+\xi_n^2)^2}(-\frac{1}{2}h'(0)\sum_{k<n}\xi_k
c(\widetilde{e_k})c(\widetilde{e_n})+\frac{5}{2}h'(0)\xi_n)-\frac{2ih'(0)\xi_n}{(1+\xi_n^2)^3}.
 \end{eqnarray}
We note that $\int_{|\xi'|=1}\xi_1\cdots\xi_{2q+1}\sigma(\xi')=0$, so the first term in (5.36) has no contribution for
computing \textbf{Case c)}.

By the Cauchy integral formula,  we get
\begin{eqnarray}
\pi_{\xi_{n}}^{+}[\frac{\xi_{n}}{(1+\xi_{n}^{2})^{2}}](x_{0})|_{|\xi'|=1}
  &=&\frac{1}{2\pi i}\lim_{u\rightarrow 0^{-}}\int_{\Gamma^{+}}\frac{\frac{\eta_{n}}{(\eta_{n}+i)^{2}(\xi_{n}+iu-\eta_{n})}}
{(\eta_{n}-i)^{2}}\texttt{d}\eta_{n}\nonumber\\
 &=&\left[\frac{\eta_{n}}{(\eta_n+i)^2(\xi_n-\eta_n)}\right]^{(1)}|_{\eta_n=i}=-\frac{i\xi_{n}}{4(\xi_{n}-i)^{2}},
 \end{eqnarray}
and
\begin{eqnarray}
\pi_{\xi_{n}}^{+}[\frac{\xi_{n}}{(1+\xi_{n}^{2})^{5}}](x_{0})|_{|\xi'|=1}
  &=&\frac{1}{2\pi i}\lim_{u\rightarrow 0^{-}}\int_{\Gamma^{+}}\frac{\frac{\eta_{n}}{(\eta_{n}+i)^{5}(\xi_{n}+iu-\eta_{n})}}
{(\eta_{n}-i)^{5}}\texttt{d}\eta_{n}\nonumber\\
 &=&\left[\frac{\eta_{n}}{(\eta_n+i)^5(\xi_n-\eta_n)}\right]^{(4)}|_{\eta_n=i}\nonumber\\
 &=&\frac{3(35+47i\xi_{n}-25\xi_{n}^{2}-5i\xi_{n}^{3})}{32(\xi_{n}-i)^{5}}.
 \end{eqnarray}
Form (5.37) and (5.38), we obtain
 \begin{equation}
\pi_{\xi_{n}}^{+}\sigma_{-3}(D^{-2})(x_{0})|_{|\xi'|=1}=-\frac{5h'(0)\xi_{n}}{8(\xi_{n}-i)^{2}}
 -\frac{3h'(0)(35+47i\xi_{n}-25\xi_{n}^{2}-5i\xi_{n}^{3})}{16(\xi_{n}-i)^{5}}.
\end{equation}
By (5.13), we have
 \begin{equation}
  \partial_{\xi_{n}}\sigma_{-1}(D^{-1})(x_{0})|_{|\xi'|=1}=\sqrt{-1}\left(\frac{c(dx_n)}
{|\xi|^2}-\frac{2\xi_nc(\xi)}{|\xi|^4}\right).
\end{equation}
Combining (5.39) and (5.40), we obtain
\begin{eqnarray}
&&\text{trace}[\pi_{\xi_{n}}^{+}\sigma_{-3}(D^{-2})
 \times \partial_{\xi_{n}}\sigma_{-1}(D^{-1})](x_{0})|_{|\xi'|=1}\nonumber\\
 &=&\left[-\frac{5h'(0)\xi_{n}}{8(\xi_{n}-i)^{2}}
 -\frac{3h'(0)(35+47i\xi_{n}-25\xi_{n}^{2}-5i\xi_{n}^{3})}{16(\xi_{n}-i)^{5}}\right]\times\left(\frac{i \texttt{tr}[c(dx_n)]}
{|\xi|^2}-\frac{2i\xi_n\texttt{tr}[c(\xi)]}{|\xi|^4}\right)\nonumber\\
 &=&0.
 \end{eqnarray}
Then
\begin{equation}
\text{Case \ c}=-i\int_{|\xi'|=1}\int_{-\infty}^{+\infty}\texttt{trace}[\partial_{\xi_{n}}\sigma_{-3}(D^{-1})
           \times\sigma_{-1}(D^{-1})](x_{0})\texttt{d}\xi_{n}\sigma(\xi')\texttt{d}x'=0.
 \end{equation}
Since $\Phi$  is the sum of the \textbf{case a, b} and \textbf{c}, so is zero. Then we have
\begin{thm}\label{th:32}
Let M be a 5-dimensional compact connected manifold with the boundary $\partial M$ and the metric $g^{M}$ as above ,
and $D$ the Dirac operator on $S(TM)$, then
\begin{equation}
Vol_{5}^{(2,1)}(M)=-\frac{\pi\sqrt[5]{30}}{240\sqrt[10]{\pi}2^p\pi^{p+\frac{q}{2}}}\int_Mr_M\texttt{d}vol,
\end{equation}
where  $r_M$ be the scalar curvature.
\end{thm}

Let us now consider the Einstein-Hilbert action for 5-dimensional manifolds with boundary.
Let $P, P'$ be two pseudodifferential operators with transmission property
(see \cite{Wa1}) on $S(TM)$. Motivated by (5.10), we define locally
\begin{eqnarray}
&&res_{2,1}(P,P'):=-\frac{1}{2}\int_{|\xi'|=1}\int_{-\infty}^{+\infty}\texttt{trace}[\partial_{x_{n}}\pi_{\xi_{n}}^{+}\sigma_{-2}(P^{-2})
 \partial_{\xi_{n}}^{2}\sigma_{-1}(P^{-1})](x_{0})\texttt{d}\xi_{n}\sigma(\xi')\texttt{d}x' ;\\
&&res_{2,2}(P,P'):=-i\int_{|\xi'|=1}\int_{-\infty}^{+\infty}\texttt{trace}[\pi_{\xi_{n}}^{+}\sigma_{-2}(P^{-2})
  \partial_{\xi_{n}}\sigma_{-2}(P^{-1})](x_{0})\texttt{d}\xi_{n}\sigma(\xi')\texttt{d}x'.
 \end{eqnarray}
From (5.26) and (5.30), we have
 \begin{equation}
\textbf{ Case a (II)}=res_{2,1}(D^{-2}, D^{-1}); \ \textbf{Case b}= res_{2,2}(D^{-2}, D^{-1}).
\end{equation}
Without loss of generality, we may assume that $\partial_{M}$ is flat, then $\{\texttt{d}x_{i}=e_{i}\}$, $g_{i,j}^{\partial_{M}}=\delta_{i,j}$,
$\partial_{x_{s}}g_{i,j}^{\partial_{M}}=0$. So $res_{2,1}(D^{-2}, D^{-1})$ and $res_{2,2}(D^{-2}, D^{-1})$
are two global forms locally defined by the aboved oriented orthonormal basis ${\texttt{d}x_{i}}$.  Hence by \textbf{Case a II)}
and \textbf{Case b},
we have
\begin{thm}\label{th:41}
Let $M$ be a 5-dimensional compact manifold with boundary $\partial_{M}$ and the metric $g^{M}$ as above, and
$D$ the Dirac operator on $S(TM)$. Assume $\partial_{M}$ is flat, then
\begin{equation}
\int_{\partial_{M}}res_{2, 1}(D^{-2}, D^{-1})=
\int_{\partial_{M}}res_{2, 2}(D^{-2}, D^{-1})=0.
\end{equation}
\end{thm}

Nextly, for $4$-dimensional spin manifolds with boundary, we compute ${\rm Vol}^{(2,1)}_4$. By (4) in \cite{Wa4}, we have
 \begin{equation}
 \widetilde{{\rm Wres}}[\pi^+D^{-2}\circ\pi^+D^{-1}]=\int_{\partial M}\Phi.
\end{equation}

When $n=4$, in (5.10), $ r-k-|\alpha|+l-j-1=-3,~~r\leq-2,l\leq-1$, so
we get $r=-2, \ l=-1,~k=|\alpha|=j=0,$ then
  \begin{equation}
\Phi=-i\int_{|\xi'|=1}\int_{-\infty}^{+\infty}\text{trace}[\pi_{\xi_{n}}^{+}
\sigma_{-2}(D^{-2})\times\partial_{\xi_{n}}\sigma_{-1}(D^{-1})](x_{0})
    \texttt{d}\xi_{n}\sigma(\xi')\texttt{d}x'.
 \end{equation}
From (5.13), we obtain
 \begin{equation}
  \partial_{\xi_{n}}\sigma_{-1}(D^{-1})(x_{0})|_{|\xi'|=1}=\sqrt{-1}\left(\frac{c(dx_n)}
{|\xi|^2}-\frac{2\xi_nc(\xi)}{|\xi|^4}\right).
\end{equation}
Combining (5.31) and (5.50), we obtain
 \begin{equation}
\text{trace}[\pi_{\xi_{n}}^{+}
\sigma_{-2}(D^{-2})\times\partial_{\xi_{n}}\sigma_{-1}(D^{-1})](x_{0})|_{|\xi'|=1}=0.
\end{equation}
Therefore ${\rm Vol}^{(2,1)}_4=0$.

\begin{rem}
In fact, we may generalize Theorem 5.4 and Theorem 5.5 to the foliation case.
\end{rem}
Then we have
\begin{thm}\label{th:32}
Let M be a 5-dimensional compact connected foliation with the boundary $\partial M$ and the metric $g^{M}$ as above ,
and $D_{F}$ be the sub-Dirac operator, then
\begin{equation}
Vol_{5}^{(2,1)}(M)=-\frac{\pi\sqrt[5]{30}}{240\sqrt[10]{\pi}2^p\pi^{p+\frac{q}{2}}}\int_Mr_M\texttt{d}vol,
\end{equation}
where  $r_M$ be the scalar curvature.
\end{thm}

\section{The Lower Volume for the Robertson-Walker Space}
One very important family of cosmological models in general relativity is the family of Robertson-Walker space-times:
\begin{equation}
L_{1}^{4}(f,c):=(I\times M, g_{f}^{c}), \ g_{f}^{c}=-\texttt{d}t^{2}+f^{2}(t)g_{c},
\end{equation}
with a warped product Lorentzian metric $g_{f}^{c}$ defined on the product of an open interval $I$ and a Riemannian 3-manifold
$(S,g_{c})$ of constant sectional curvature $c$.

Let $\tilde{M}=I\times _{f}M$ be a Riemannian manifold with the metric $g_{f}=\texttt{d}t^{2}+f^{2}(t)g^{TM}$. Now we compute the
lower dimension volumes for 4-dimensional foliations $\tilde{M}$ with spin leave $I$.
Let $\mathcal{L}(I)$ and $\mathcal{L}(M)$ be the set of lifts of vector fields on $I$ and $M$ to $I\times_{f}M$ respectively.
For $q\in M$, the horizontal leaf $\eta^{-1}(M)$ is a totally geodesic submanifold isometric to $I$ with scalar factor $1/h(p)$.
For $p\in I$, $\pi^{-1}(p)$ is a totally umbilical submanifold that is homothetically isomorphic to $M$ with scalar factor $1/f(p)$.
The submanifolds $\pi^{-1}(p)=\{p\}\times F$, $p\in I$ and $\eta^{-1}(q)=I\times\{q\}$, $q\in M$ are called fibers and leaves respectively.
A vector field on $M$ is called vertical if it is always tangent to fibers; and horizontal if it is always orthogonal to fibers.
We use the corresponding terminology for individual tangent vectors as well.

Let $\mathcal{H}$ and $\mathcal{V}$ denote the projections of tangent spaces of $\tilde{M}$ onto the subspaces of horizontal and
vertical vectors, respectively. We use the same letters to denote the horizontal and vertical distributions.
On the  warped product $I\times_{f}M$, denote by $\partial_{t}$ the lift of the standard vector field $d/dt$ on $I$ of $I\times_{f}M$,
so we have $\partial_{t}\in\mathcal{L}(I)$.

For a vector field $V$ on $I\times _{f}M$, we decompose $V$ into a sum
\begin{equation}
V=\varphi_{V}\partial_{t}+\hat{V},
\end{equation}
where $\varphi_{V}=\langle V, \partial_{t}\rangle$ and $\hat{V}$ is the vertical component of $V$ that is orthogonal to $\partial_{t}$.
By Proposition 2.2 and Proposition 2.4 in \cite{DB}, we have

\begin{lem}
Let $\tilde{M}=I\times _{f}M$ be a Riemannian manifold with the metric $g_{f}=\texttt{d}t^{2}+f^{2}(t)g$.
For vector fields $X,Y$ in $\mathcal{L}(M)$, then
\begin{eqnarray}
&&(1) \ \tilde{\nabla}_{\partial_{t}}\partial_{t}=0,\\
&&(2) \ \tilde{\nabla}_{\partial_{t}} X=\tilde{\nabla}_{X}\partial_{t}=(ln f)'X,\\
&&(3) \ \nabla_{X}Y=\nabla_{X}^{M}Y-\frac{g(X,Y)}{f}grad (f).
\end{eqnarray}
\end{lem}

\begin{lem}
For vector fields $X,Y, Z$ in $\mathcal{L}(M)$, the curvature tensor $\tilde{R}$ of $\tilde{M}$ satisfies
\begin{eqnarray}
&&(1) \ \tilde{R}(\partial_{t},X)\partial_{t}=\frac{f''}{f}X,\\
&&(2) \  \tilde{R}(X,\partial_{t})Y=\langle X, Y\rangle\frac{f''}{f}\partial_{t},\\
&&(3) \ \tilde{R}(X,Y)\partial_{t}=0,\\
&&(4) \ \tilde{R}(X,Y)Z=R^{M}(X,Y)Z-\frac{\langle grad (f), grad (f) \rangle}{f^{2}}\{\langle X, Z\rangle Y-\langle Y, Z\rangle X\}.
\end{eqnarray}
\end{lem}

Let $M$ be a foliation with boundary $\partial M$. Let $\psi\in \Gamma(S(F)\otimes\wedge(F^{\perp,\star}))$,
 We impose the Dirichlet boundary conditions $\psi|_{\partial M}=0$. With the Dirichlet boundary conditions in \cite{BG}, we have the
 heat trace
asymptotics for $t\rightarrow 0$
$$\texttt{tr}(e^{-tD_F^2})\sim \sum_{n\geq0}t^{\frac{n-m}{2}}a_{n}(D_F^2).$$
When ${\rm dim} \tilde{M}=4$, by (18) in \cite{ILV},
one uses the Seely-deWitt
coefficients $a_{n}(D_F^2)$ and $t=\wedge^{-2}$ to obtain an
asymptotics for the spectral action
\begin{eqnarray}
I&=&\texttt{tr}\widehat{F}\left(\frac{D^2_F}{\wedge^2}\right)\sim\wedge^4F_4a_0(D^2_F)+\wedge^3F_3a_1(D^2_F)\nonumber\\
 &&+\wedge^2F_2a_2(D^2_F)+\wedge F_1a_3(D^2_F)+\wedge^0F_0a_4(D^2_F)~~{\rm as} ~~\wedge\rightarrow \infty,
\end{eqnarray}
where
\begin{eqnarray}
F_k:=\frac{1}{\Gamma(\frac{k}{2})}\int_0^{\infty}\widehat{F}(s)s^{\frac{k}{2}-1}\texttt{d}s.
\end{eqnarray}
Let $N=e_m$ be the inward pointing unit normal vector on $\partial
\tilde{M}$ and $e_i, 1\leq i\leq m-1$ be the orthonormal frame on
$T(\partial M)$. Let $ L_{ab}=(\nabla_{e_a}e_b,N)$ be the second
fundamental form and indices $\{a,b,\cdots\}$ range from $1$ through
$m-1$. By Theorem 1.1 in \cite{BG}, we obtain the first five coefficients of
the heat trace asymptotics
\begin{eqnarray}
a_0(D_F^{2})&=&(4\pi)^{-\frac{m}{2}}\int_{\tilde{M}}\texttt{tr}(\texttt{Id})\texttt{d}vol_{\tilde{M}},\\
a_1(D_F^{2})&=&-4^{-1}(4\pi)^{-\frac{(m-1)}{2}}\int_{\partial {\tilde{M}}}\texttt{tr}(\texttt{Id})\texttt{d}vol_{\partial {\tilde{M}}},\\
a_2(D_F^{2})&=&(4\pi)^{-\frac{m}{2}}6^{-1}\Big\{\int_{\tilde{M}}\texttt{tr}(r_{\tilde{M}}+6E)\texttt{d}vol_{\tilde{M}}+2\int_{\partial \tilde{M}}
{\rm tr}(L_{aa})\texttt{d}vol_{\partial {\tilde{M}}}\Big\},\\
a_3(D_F^{2})&=&-4^{-1}(4\pi)^{-\frac{(m-1)}{2}}96^{-1}\{\int_{\partial {\tilde{M}}}\texttt{tr}(96E+16r_{\tilde{M}}+8R_{aNaN}+7L_{aa}L_{bb}    \nonumber\\
&&-10L_{ab}L_{ab})\texttt{d}vol_{\partial \tilde{M}})\},\\
a_4(D_F^{2})&=&\frac{(4\pi)^{-\frac{m}{2}}}{360}\{\int_{\tilde{M}}\texttt{tr}[-12R_{ijij,kk}+5R_{ijij}R_{klkl}
   -2R_{ijik}R_{ljlk}+2R_{ijkl}R_{ijkl}\nonumber\\
&&-60R_{ijij}E+180E^2+60E_{,kk}+30\Omega_{ij}
     \Omega_{ij}]\texttt{d}vol_{\tilde{M}} \nonumber\\
&&+\int_{\partial {\tilde{M}}}\texttt{tr}(-120E;_N-18r_{\tilde{M}};_N+120EL_{aa}+20r_{\tilde{M}}L_{aa}+4R_{aNaN}L_{bb} \nonumber\\
&&-12R_{aNbN}L_{ab}+4R_{abcb}L_{ac}+24L_{aa;bb}+40/21L_{aa}L_{bb}L_{cc}\nonumber\\
&&-88/7L_{ab}L_{ab}L_{cc}+320/21L_{ab}L_{bc}L_{ac})\texttt{d}vol_{\partial {\tilde{M}}}\}.
\end{eqnarray}
By (2.5), (2.19) and the divergence theorem for manifolds with boundary, we obtain
\begin{eqnarray}
a_0(D_F)&=&\frac{1}{2^p\pi^{p+\frac{q}{2}}}\int_{\tilde{M}}\texttt{d}vol_{\tilde{M}},\\
a_1(D_F)&=&-4^{-1}(4\pi)^{-\frac{(m-1)}{2}}2^{p+q}\int_{\partial \tilde{M}}\texttt{d}vol_{\partial \tilde{M}},\\
a_2(D_F)&=&\frac{1}{12\cdot 2^p\pi^{p+\frac{q}{2}}}(-\int_{\tilde{M}} r_{\tilde{M}} dvol_{\tilde{M}}+4\int_{\partial
\tilde{M}}L_{aa}\texttt{d}vol_{\partial \tilde{M}}),\\
a_3(D_F)&=&-4^{-1}(4\pi)^{-\frac{(m-1)}{2}}96^{-1}2^{p+q}\Big\{\int_{\partial \tilde{M}}(-8r_{\tilde{M}}+8R_{aNaN}+7L_{aa}L_{bb}\nonumber\\
&&-10L_{ab}L_{ab})\texttt{d}vol_{\partial \tilde{M}}\Big\},\\
a_4(D_F)&=&\frac{(4\pi)^{-\frac{m}{2}}}{360}2^{p+q}\Big\{\int_{\tilde{M}}\left(\frac{5}{4}r_{\tilde{M}}^2-2R_{ijik}R_{ljlk}-\frac{7}{4}
    R_{ijkl}^2+\frac{15}{2}||R^{F^\bot}||^2\right)\texttt{d}vol_{\tilde{M}}\nonumber\\
&&+\int_{\partial \tilde{M}}{\rm tr}(-51r_{{\tilde{M}};N}-10r_{\tilde{M}}L_{aa}+4R_{aNaN}L_{bb}-12R_{aNbN}L_{ab}\nonumber\\
&&+4R_{abcb}L_{ac}+24L_{aa;bb}+40/21L_{aa}L_{bb}L_{cc}-88/7L_{ab}L_{ab}L_{cc}\nonumber\\
&&+320/21L_{ab}L_{bc}L_{ac})\texttt{d}vol_{\partial \tilde{M}}\Big\}.
\end{eqnarray}

Consider $\tilde{M}=I\times _{f}M$ be a Riemannian manifold with the metric $g_{f}=\texttt{d}t^{2}+f^{2}(t)g^{TM}$. As in \cite{Wa3},
we take normal coordinates in boundary and we get orthonormal frame $\{\partial_{t}, e_1,e_2,e_{3}\}$. In the sequel we let
$L_{aa}=\langle \nabla_{e_{a}}e_{a}, \partial_{t}\rangle$ be the second fundamental form and
$\tilde{R}_{ijik}=\langle \tilde{R}(e_{i},e_{j})e_{k}, e_{l}\rangle$ be the components of the curvature tensor in local coordinates in \cite{RP}.
Then we obtain
\begin{equation}
L_{aa}=\langle \nabla_{e_{a}}e_{a}, \partial_{t}\rangle=-\delta_{a}^{a}(lnf)'=-3(lnf)',
\end{equation}
and
\begin{equation}
\tilde{R}_{aNaN}=\langle \tilde{R}(e_{a},\partial_{t})e_{a}, \partial_{t}\rangle=3\frac{f''}{f}.
\end{equation}
Similarly we obtain
\begin{eqnarray}
&&L_{bb}=-3(lnf)'; \  L_{aa}L_{bb}=9(\frac{f'}{f})^{2}; \  L_{ab}L_{ab}=3(\frac{f'}{f})^{2}; \nonumber\\
&&L_{aa}L_{bb}L_{cc}=-27(\frac{f'}{f})^{3}; \ L_{ab}L_{ab}L_{cc}=-9(\frac{f'}{f})^{3}; \ L_{ab}L_{bc}L_{ac}=-3(\frac{f'}{f})^{3}; \nonumber\\
&&\tilde{R}_{ijik}\tilde{R}_{ljlk}=\tilde{R}_{ijik}^{M}\tilde{R}_{ljlk}^{M}+12(\frac{f''}{f})^{2}; \nonumber\\
&&(\tilde{R}_{ijkl})^{2}=(\tilde{R}^{M}_{ijkl})^{2}+12(\frac{f''}{f})^{2}; \ \tilde{R}_{aNaN}L_{ab}=-3\frac{f'f''}{f^{2}} ; \nonumber\\
&&\tilde{R}_{aNaN}L_{bb}=-9\frac{f'f''}{f^{2}} ;  \ \tilde{R}_{abcb}L_{ac}=3\frac{f'}{f}r_M-18\frac{f'f''}{f^{2}};\nonumber\\
&&r_{\tilde{M}}=\frac{r_M}{f^{2}}+6(\frac{f''}{f}+\frac{(f')^{2}}{f^{2}}); \nonumber\\
&&||R^{F^\bot}||^2=\sum_{s,t,r,l=1}^q(R^{M}_{rlts})^{2}.
\end{eqnarray}
Then we obtain

\begin{thm}
Let $\tilde{M}=I\times _{f}M$  be  a compact 4-dimensional oriented foliation with spin leave, then the spectral action for  sub-Dirac operators
\begin{eqnarray}
a_0(D_F)&=&\frac{1}{2^p\pi^{p+\frac{q}{2}}}\int_{\tilde{M}}\texttt{d}vol_{\tilde{M}},\\
a_1(D_F)&=&-4^{-1}(4\pi)^{-\frac{(m-1)}{2}}2^{p+q}\int_{\partial \tilde{M}}\texttt{d}vol_{\partial \tilde{M}},\\
a_2(D_F)&=&\frac{1}{12\cdot 2^p\pi^{p+\frac{q}{2}}}\Big[-\int_{\tilde{M}}\Big(\frac{r_M}{f^{2}}+6(\frac{f''}{f}+\frac{(f')^{2}}{f^{2}})\Big)dvol_{\tilde{M}}-12\int_{\partial
\tilde{M}}(lnf)'\texttt{d}vol_{\partial \tilde{M}}\Big],\\
a_3(D_F)&=&-\frac{1}{384}(4\pi)^{-\frac{(m-1)}{2}}2^{p+q}
\int_{\partial \tilde{M}}\Big(-8\frac{r_M}{f^{2}}-24\frac{f''}{f}-15(\frac{f'}{f})^{2}\Big)\texttt{d}vol_{\partial \tilde{M}},\\
a_4(D_F)&=&\frac{(4\pi)^{-\frac{m}{2}}}{360}2^{p+q}\Big\{\int_{\tilde{M}}\Big[\frac{5}{4}
\Big(\frac{r_M}{f^{2}}+6(\frac{f''}{f}+\frac{(f')^{2}}{f^{2}})\Big)^{2}
-2R^{M}_{ijik}R^{M}_{ljlk}+\frac{23}{4}(R^{M}_{ijkl})^{2}-45(\frac{f''}{f})^{2}\Big]\texttt{d}vol_{\tilde{M}}\nonumber\\
&&+\int_{\partial \tilde{M}}{\rm tr}\Big((\frac{102}{f^{3}}+\frac{30}{f^{2}}+\frac{12f'}{f})r_M-306\frac{ff'''}{f^{2}}-378\frac{f'f''}{f^{2}}
+180(\frac{f'}{f})^{2}  \nonumber\\
&&+180\frac{f''}{f}+628(\frac{f'}{f})^{3}\Big)\texttt{d}vol_{\partial \tilde{M}}\Big\}.
\end{eqnarray}
\end{thm}

Nextly, Consider $\bar{M}=S^{1}\times _{f}M^{n}$ be a Riemannian manifold with the metric $g_{f}=\texttt{d}t^{2}+f^{2}(t)g^{TM}$.
As in \cite{Wa3}, we take normal coordinates in boundary and we get orthonormal frame $\{\partial_{t}, e_1,e_2,e_{3}\}$.
By (2.34), we get $||R^{F^\bot}||^2=\sum_{s,t,r,l=1}^q(R^{M}_{rlts})^{2}$.
Then by (6.24) and Theorem 2.9, we obtain

\begin{thm}
Let $\bar{M}=S^{1}\times _{f}M^{n}$  be  a Robertson-Walker space, then
\begin{eqnarray}
 Vol_{(n+1,1)}^{(n-3)}(\bar{M},F)&=& \frac{v_{n+1,n-3}}{360\cdot2^p\pi^{p+\frac{q}{2}}}\int_{\tilde{M}}\Big(\frac{5}{4}
\Big(\frac{r_M}{f^{2}}+6(\frac{f''}{f}+\frac{(f')^{2}}{f^{2}})\Big)^{2}-2R^{M}_{ijik}R^{M}_{ljlk}\nonumber\\
&&+\frac{23}{4}(R^{M}_{ijkl})^{2}-45(\frac{f''}{f})^{2}\Big)\texttt{d}vol_{\tilde{M}};  \\
 Vol^{(n-1)}_{(n+1,1)}(M,F)&=& -\frac{v_{n+1,n-1}}{12\cdot2^p\pi^{p+\frac{q}{2}}}\int_{\tilde{M}}
\Big(\frac{r_M}{f^{2}}+6(\frac{f''}{f}+\frac{(f')^{2}}{f^{2}})\Big)\texttt{d}vol_{\tilde{M}};  \\
 Vol^{(n+1)}_{(n+1,1)}(M,F)&=& \frac{v_{n+1,n+1}}{2^p\pi^{p+\frac{q}{2}}}\int_{\tilde{M}}f^{3}\texttt{d}vol_{\tilde{M}} .
\end{eqnarray}
when $k$ is even and $n$ is even, $v_{n,k}=\frac{k}{n}(2\pi)^{\frac{k-n}{2}}\frac{\Gamma(\frac{n}{2}+1)^{\frac{k}{n}}}{\Gamma(\frac{k}{2}+1)};$
when $k$ is odd and $n$ is odd, $v_{n,k}=\frac{k}{n}2^{\frac{(k-n)(n+1)}{2n}}\times\pi^{\frac{k-n}{2}}\frac{\Gamma(\frac{n}{2}+1)^{\frac{k}{n}}}
{\Gamma(\frac{k}{2}+1)}.$
\end{thm}

\section*{ Acknowledgements}
This work was supported by Fok Ying Tong Education Foundation under Grant No. 121003 and NSFC. 11271062. The author also thank the referee
for his (or her) careful reading and helpful comments.

\end{document}